\documentclass[12pt]{article}
\usepackage{amsmath,amssymb,amsfonts,amsthm}
\newenvironment{myabstract}{\par\noindent
{\bf Abstract . } \small }
{\par\vskip8pt minus3pt\rm}
\newcounter{item}[section]
\newcounter{kirshr}
\newcounter{kirsha}
\newcounter{kirshb}
\newenvironment{enumroman}{\setcounter{kirshr}{1}
\begin{list}{(\roman{kirshr})}{\usecounter{kirshr}} }{\end{list}}
\newenvironment{enumarab}{\setcounter{kirshb}{1}
\begin{list}{(\arabic{kirshb})}{\usecounter{kirshb}} }{\end{list}}
\newenvironment{athm}[1]{\vskip3mm\par\noindent
{\bf #1 }. \slshape }
{\upshape\par\vskip10pt minus3pt}
\newtheorem{theorem}{Theorem}[section]

\newtheorem{lemma}[theorem]{Lemma}
\newtheorem{corollary}[theorem]{Corollary}
\newenvironment{demo}[1]{\noindent{\bf #1.}\upshape\mdseries}
{\nopagebreak{\hfill\rule{2mm}{2mm}\nopagebreak}\par\normalfont}
\theoremstyle{definition}

\newtheorem{example}[theorem]{Example}
\newtheorem{definition}[theorem]{Definition}
\newtheorem{question}[theorem]{question}
\def\R{\mathbb{R}}
\def\Q{\mathbb{Q}}
\def\C{{\mathfrak{C}}}
\def\Fm{{\mathfrak{Fm}}}

\def\Fr{{\mathfrak{Fr}}}
\def\Sg{{\mathfrak{Sg}}}
\def\Fm{{\mathfrak{Fm}}}
\def\A{{\mathfrak{A}}}
\def\B{{\mathfrak{B}}}
\def\Bl{{\mathfrak{Bl}}}
\def\C{{\mathfrak{C}}}

\def\M{{\mathfrak{M}}}
\def\N{{\mathfrak{N}}}

\def\Rd{{\ Rd}}
\def\(R)RA{{\bf (R)RA}}

\def\R{\mathbb{R}}
\def\Q{\mathbb{Q}}
\def\B{{\sf B}}
\def\G{{\sf G}}

\def\tp{{\sf tp}}
 \def\Cm{{\mathfrak{Cm}}}

\def\F{{\mathfrak F}}
\def\At{{\sf At}}
\def\Rl{{\mathfrak Rl}}

\def\Tm{{\mathfrak{Tm}}}
\def\A{{\mathfrak{A}}}
\def\B{{\mathfrak{B}}}
\def\C{{\mathfrak{C}}}

\def\A{{\mathfrak{A}}}
\def\B{{\mathfrak{B}}}
\def\C{{\mathfrak{C}}}

\def\Uf{{\mathfrak{Uf}}}
\def\Cm{{\mathfrak Cm}}
\def\F{{\mathfrak F}}
\def\H{{\mathfrak H}}
\def\G{{\mathfrak G}}

\def\L{{\mathfrak{L}}}
\def\Ig{{\mathfrak Ig}}
\def\Str{{\mathfrak Str}}

\def\L{{\mathfrak{L}}}

\def\G{{\bf G}}

\def\Diag{{\bf Diag}}

\def\CM{{\bf CM}}

\title{On complete representability of Pinter's algebras and related structures}
\author{Tarek Sayed Ahmed}

\begin{document}
\maketitle

\begin{myabstract} We answer an implicit question of Ian Hodkinson's. 
We show that atomic Pinters algebras may not be completely representable, however 
the class of completely representable Pinters algebras is elementary and finitely axiomatizable. We obtain 
analagous results for infinite dimensions (replacing finite axiomatizability by finite schema axiomatizability). 
We show that the class of subdirect products of set algebras is a canonical variety that is locally finite only for finite dimensions, 
and has the superamalgamation property; the latter for all dimensions.
However, the algebras we deal with are expansions of Pinter algebras 
with substitutions corresponding to tranpositions. It is true that this makes  the a lot of the problems addressed harder, but this is an acet, not a liability.
Futhermore,  the results for Pinter's algebras readily follow by just discarding the substitution operations corresponding to transpostions.
Finally,  we show that the multi-dimensional modal logic corresponding to finite dimensional algebras have an $NP$-complete satisfiability problem. 

\end{myabstract}

\section {Introduction}

Suppose we have a class of algebras infront of us. The most pressing need is to try and classify it. Classifying is a kind of defining.
Most mathematical classification is by axioms, either first order, or even better equations. 
In algebraic logic the typical question is this. Given a class of concrete set algebras, that we know in advance is elementary or is a variety.
Furthermore, such algebras consist of sets of sequences 
(usually with the same length called the dimension) 
and the operations are set - theoretic, utilizing the form of elements as sets of sequences.
Is there a {\it simple}  elementary (equational) axiomatization of this class?
A harder problem is: Is their a {\it finite}  elementary (equational) axiomatization of this class?

The prime examples of such operations defined on the unit of the algebra, which is in the form of  
$^nU$ ($n\geq 2)$ are cylindrifiers and substitutions.
For $i<n$, and $t, s\in {}^nU$, define, the equivalence relation,  $s\equiv_i t$ if $s(j)=t(j)$ for all $j\neq i$. Now fix $i<n$ and $\tau\in {}^nn$, then 
these operations are defined as follows
$$c_iX=\{s\in {}^nU: \exists t\in X, t \equiv_i s\},$$
$$s_{\tau}X=\{s\in {}^nU: s\circ \tau\in X\}.$$
Both are unary operations on $\wp(^nU)$; the $c_i$ is called the $i$th cylindrfier, while the $s_{\tau}$ is called the substitution operation corresponding to 
te transformation $\tau$, or simply a substitution. 

For Boolean algebras this question is completely settled by Stone's representation theorem.
Every Boolean algebra is representable, equivalently, the class of Boolean algebras is 
finitely axiomatizable by a set of equations.
This is equivalent to the completeness of propositional logic.

When we approach the realm of first order logic things tend to become much more complicated. 
The standard algebraisation of first order logic is cylindric algebras (where cylindrifiers are the prominent citizens)
and polyadic algebras (where cylindrifiers and substitutions are the prominent citizens).
Such algebras, or rather the abstract version thereof,  are defined by a finite set of equations that aim to capture algebraically the properties of cylindrifiers and substitutions
(and diagonal elements if present in the signature).

Let us concentrate on polyadic algebras of dimension $n$; where $n$ is a finite ordinal.  A full set algebra is one whose unit is of the form $^nU$ 
and the non-Boolean operations are cylindrifiers and substitutions.
The class of representable algebras, defined as the class of subdirect products of full set algebras is a 
discriminator variety that is not finitely axiomatizable for $n\geq 3$, thus the set of equations postulated by Halmos is not complete.
Furthermore, when we also have diagonal elements, then there is an inevitable degree of complexity in any potential universal axiomatization.

There is another type of  representations for polyadic algebras, and that is {\it complete} representations. 
An algebra is completely representable if it has a representation that preserves arbitrary meets whenever they exist.
For Boolean algebras the completely representable  algebras are  easily characterized; 
they are simply the atomic ones; in particular, this class is elementary and finitey axiomatizable, one just adds the first order sentence
expressing atomicity.
For cylindric and polyadic algebras, again, this problem turns much more involved,
This class for $n\geq 3$ is not even elementary.

Strongly related to complete representations \cite{Tarek}, is the notion of omitting types for the corresponging multi-dimensional modal logic. 
Let $W$ be a class of algebras (usually a variety or at worst quasi-variety) with a Boolean reduct, having
the class $RW$ as the class of representable algebras, so that $RW\subseteq W$, and for $\B\in RW$, 
$\B$ has top element a set of sequences having the same 
length, say $n$ (in our case the dimension of the algebra), 
and the Boolean operations are interpreted as concrete intersections and complementation of 
$n$-ary relatons. 
We say the $\L_V$, the multi-dimensional modal logic has the omitting types theorem, if whenever $\A\in V$ is countable, 
and $(X_i: i\in \omega)$ is a family of non-principal types, meaning that $\prod X_i=0$ for each $i\in \omega$, then there is a $\B\in RW$ with unit $V$,
and an injective homomorphism 
$f:\A\to \wp(V)$ such that $\bigcap_{i\in \omega}f(X_i)=\emptyset$ for each $i\in \omega$.

In this paper we study, among other things, 
complete representability for cylindrifier free reducts of polyadic algebras, as well as omitting types for the corresponding multi-dimensional 
model logic.
We answer a question of Hodkinson \cite{atomic} p.  by showing that for various such reducts of polyadic algebras, 
atomic algebras might not be completely representable, however, 
they can be easily characterized by a finite simple set of first order formulas.


Let us describle our algebras in a somewhat more general setting.
Let $T$ be a submonoid of $^nn$ and $U$ be a non-empty set. 
A set $V\subseteq {}^nU$ is called $T$ closed, if whenever $s\in V$ and $\tau\in T$, 
then $s\circ \tau\in V$. (For example $T$ is $T$ closed).
If $V$ is $T$ closed then $\wp(V)$ denotes the set algebra $({\cal P}(V),\cap,\sim s_{\tau})_{\tau\in T}$.
$\wp(^nU)$ is called a full set algebra.

Let $GT$ be a set of generators of $T$.
One can obtain a variety $V_T$ of Boolean algebras with extra non-boolean operators $s_{\tau}$, $\tau\in GT$ 
by  translating a presentation of $T$, via the finite set of generators $GT$ to equations, and stipulating that the $s_{\tau}$'s 
are Boolean endomorphisms. 

It is known that every monoid  not necessarily finite, 
has a presentation. For finite monoids, the multiplicative table provides one. 
Encoding finite presentations in terms of a set of generators of $T$ 
into a finite set of equations $\Sigma$, enables one
to define for each $\tau\in T$, a substitution unary operation $s_{\tau}$  
and for any algebra $\A$, such that  $\A\models \Sigma$, $s_{\tau}$ is a Boolean endomorpsim of $\A$ and for $\sigma, \tau\in T$,
one has $s_{\sigma}^{\A}\circ s_{\tau}^{\A}(a)=s_{\sigma\circ \tau}^{\A}(a)$ for each $a\in A$.

The translation of presentations to equations, guarantee that
if $\A\models \Sigma$ and $a\in \A$ is non zero, then for any Boolean ultrafilter $F$ of $\A$ containing $a$,
the map $f:\A\to \wp(T)$ defined via 
$$x\mapsto \{\tau \in T: s_{\tau}x\in F\}$$ is a homomorphism 
such that $f(a)\neq 0$. 
Such a homomorphism determines a (finite) relativized representation meaning that the unit of the set algebra is possibly only a proper subset
of the square $^nn$. 

Let $RT_n$ be the class of subdirect products of {\it full} set algebras; those set algebras whose units are squares 
(possibly with infinite base). One can show that $\Sigma$ above axiomatizes the variety generated by $RT_n$, but it is not obvious
that $RK_n$ is closed under homomorphic images. 

Indeed, if $T$ is the monoid of all non-bijective maps, that $RTA_n$ is only a quasi-variety.
Such algebras are called Pinters algebras. Sagi \cite{sagiphd}  studied the representation problem for such algebras.
In his recent paper \cite{atomic}, Hodkinson asks whether atomic such algebras are completely representable. 

In this paper we answer Hodkinson's question completely; but we deal with the monoid $T={}^nn$, with transpositions and replacements as a set of 
generators; all our results apply to Pinter's algebras. In particular, we show that atomic algebras 
are not necessarily completely representable, but that the class of completely representable algebras
is far less complex than in the case  when we have cyindrifiers, like cylindric algebras. It turns out that  
this class is finitely axiomatizable in first order logic by a very simple set
of first order sentences, expressing additivity of the extra non-boolean operations, namely, the 
substitutions. Taking the appropriate reduct we answer Hodkinson's question formulated for Pinter's algebras.

We also show that this variety is locally finite and has the superamalgamaton property.
All results except for local finiteness are proved to hold for infinite dimensions.

We shall always deal with a class $K$ of Boolean algebras with operators. We shall denote its corresponding multi-dimensional modal lgic by $\L_K$.

\section{Representability}

Here we deal with algebras, where substitutions are indexed by transpositions and replacements, so 
that we are dealing with the full monoid $^nn$.
A transpostion that swaps $i$ and $j$ will be denoted by $[i,j]$ and the replacement that take $i$ to $j$ and leaves everything else 
fixed will be denoted by $[i|j]$.The treatment resembles closely Sagi's \cite{sagiphd}, with one major difference, and that is we prove that the class of subdirect product of full set 
algebras is a variety (this is not the case with Pinter's algebras).

\begin{definition}[\emph{Substitution Set Algebras with Transpositions}]
Let $U$ be a set. \emph{The full substitution set algebra with transpositions of dimension} $\alpha$ \emph{with base} $U$ is the algebra
$$\langle\mathcal{P}({}^\alpha U); \cap,\sim,S^i_j,S_{ij}\rangle_{i\neq j\in\alpha},$$
where  the $S_i^j$'s and  $S_{ij}$'s are unary operations defined
by $$S_{i^j}(X)=\{q\in {}^\alpha U:q\circ [i|j]\in X\},$$
and $$S_{ij}(X)=\{q\in {}^{\alpha}U: q\circ [i,]\in X\}.$$
The class of \emph{Substitution Set Algebras with Transpositions of dimension} $\alpha$ is defined as follows:
$$SetSA_\alpha=\mathbf{S}\{\A:\A\text{ is a full substitution set algebra with transpositions} $$
$$\text{of dimension }\alpha  \text{ with base }U,\text{ for some set }U\}.$$
\end{definition}

The full set algebra $\wp(^{\alpha}U)$ can be viewed as the complex algebra of the atom structure or the modal frame 
$(^{\alpha}U, S_{ij})_{i,j\in \alpha}$ where
for all $i,j, S_{ij}$ is an accessibility  binary relation, such that for $s,t\in {}^{\alpha}U$, $(s,t)\in S_{ij}$ iff $s\circ [i,j]=t.$ 
When we consider arbitrary subsets of the square $^nU$,
then from the modal point of view we are restricting or relativizing the states or assignments to $D$.
On the other hand, subalgebras of full set algebras, can be viewed as {\it general} modal frames, 
which are $BAO$'s and ordinary frames, rolled into one. 

In this context, if one wants to use traditional terminology from modal logic, this means that the assignments are {\it not} 
links between the possible (states) worlds of the model; they {\it themselves} 
are the possible (states) worlds.

\begin{definition}[\emph{Representable Substitution Set Algebras with Transpositions}]
The class of {\it representable substitution set algebras with transpositions of dimension $\alpha$ }is defined to be 
$$RSA_\alpha=\mathbf{SP}SetSA_\alpha.$$
\end{definition}

\begin{definition}[\emph{locally square Set}]
Let $U$ be a given set, and let $D\subseteq{}^\alpha U.$ We say that $D$ is \emph{locally square} iff it satisfies the following condition:
$$(\forall i\neq j\in\alpha)(\forall s\in{}^\alpha U)(s\in D\Longrightarrow s\circ [i/j]\mbox{ and }s\circ [i,j]\in D),$$

\end{definition}

\begin{definition}[\emph{locally square Algebras}]
The class of \emph{locally square Set Algebras} of dimension $\alpha$
is defined to be
$$WSA_{\alpha}=\mathbf{SP}\{\langle\mathcal{P}(D); \cap,\sim,S^i_j,S_{ij}\rangle_{i\neq j\in \alpha}: U\text{ \emph{is a set}},
D\subseteq{}^\alpha U\text{\emph{permutable}}\}.$$
Here the operatins are relatvized to $D$, namely 
$S_j^i(X)=\{q\in D:q\circ [i/j]\in X\}$ and $S_{ij}(X)=\{q\in D:q\circ [i,j]\in X\}$, and $\sim$ is complement w.r.t. $D$.\\
If $D$ is a locally square set then the algebra $\wp(D)$
is defined to be $$\wp(D)=\langle\mathcal{P}(D);\cap,\sim,S^i_j,S_{ij}\rangle_{i\neq j\in n}.$$
\end{definition}

It is easy to show:
\begin{theorem}\label{relativization}
Let $U$ be a set and suppose $G\subseteq{}^n U$ is
locally square. Let $\A=\langle\mathcal{P}({}^n U);\cap,\sim,S^i_j,S_{ij}\rangle_{i\neq j\in n}$ and
let $\mathcal{B}=\langle\mathcal{P}(G);\cap,\sim,S^i_j,S_{ij}\rangle_{i\neq j\in n}$.
Then the following function $h$ is a homomorphism. $$h:\A\longrightarrow\mathcal{B},\quad h(x)=x\cap G.$$
\end{theorem}
\begin{proof} Straigtforward from the definitions.
\end{proof}
\begin{definition}[\emph{Small algebras}]
For any natural number $k\leq n$ the algebra
$\A_{nk}$ is defined to be $$\A_{nk}=\langle\mathcal{P}({}^nk);\cap,\sim,S^i_j,S_{ij}\rangle_{i\neq j\in n}.$$ So $\A_{nk}\in SetSA_n$.
\end{definition}

\begin{theorem}
$RSA_n=\mathbf{SP}\{\A_{nk}:k\leq n\}.$
\end{theorem}
\begin{proof} Exactly like the proof in \cite{sagiphd} for Pinter's algebras, however we include the proof for self completeness.
Of course, $\{\A_{nk}:k\leq n\}\subseteq RSA_n,$ and since, by definition,
$RSA_n$ is closed under the formation of subalgebras and direct products, $RSA_n\supseteq\mathbf{SP}\{\A_{nk}:k\leq n\}.$

To prove the other slightly more difficult inclusion, it is enough to show $SetSA_n\subseteq \mathbf{SP}\{\A_{nk}:k\leq n\}.$
Let $\A\in SetSA_n$ and suppose that $U$ is the base of $\A.$
If $U$ is empty, then $\A$ has one element, and one can easily show $\A\cong\A_{n0}.$
Otherwise for every $0^\A\neq a\in A$ we can construct a homomorphism $h_a$ such that
$h_a(a)\neq 0$ as follows. If $a\neq 0^\A$ then there is a sequence $q\in a.$ Let  $U_0^a=range(q)$.
Clearly, $^nU_0^a$ is locally square and herefore by theorem \ref{relativization} relativizing by $^nU_0^a$ is a homomorphism to
$\A_{nk_a}$ (where $k_a:=|range(q)|\leq n$). Let $h_a$ be this homomorphism. Since $q\in {}^nU_0^a$ we have $h_a(a)\neq0^{\A_{nk_a}}.$
One readily concludes that $\A\in\mathbf{SP}\{\A_{nk}:k\leq n\}$ as desired.
\end{proof}


\subsection{Axiomatizing $RSA_n$.}

We know that the variety generated by $RTA_n$ is finitely axiomatizable since it is generated by finitely many finite algebras,
and because, having a Boolean reduct, it is congruence distributive. This follows from a famous theorem by Baker.

In this section we show that $RSA_n$ is a variety by providing  a particular 
finite set $\Sigma_n$ of equations such that
$\mathbf{Mod}(\Sigma_n)=RSA_n$.
We consider the similarity types $\{., -, s_i^j, s_{ij}\}$, where $.$ is the Boolean meet, $-$ 
is complementation and for $i,j\in n$, $s_i^j$ and $s_{ij}$ are unary
operations, designating substitutions. We consider meets and complementation are the basic operation and $a+b$ abbreviates
$-(-a.-b).$ Our choice of equations is not haphazard; we encode a presentation of the semigroup $^nn$ into the equations, 
and further stipulate that the substitution operations are Boolean 
endomorphisms. We chose the presentation given in \cite{semigroup}
 
\begin{definition}[\emph{The Axiomatization}]\label{ax2}
For all natural $n>1$, let $\Sigma'_n$ be the following set of equations
joins. For distinct $i,j,k,l$
\begin{enumerate}
\item The Boolean axioms
\item $s_{ij}$ preserves joins and meets
\item$s_{kl}s^j_is_{kl}x=s^j_ix$
\item$s_{jk}s^j_is_{jk}x=s^k_ix$
\item$s_{ki}s^j_is_{ki}x=s^j_kx$
\item$s_{ij}s^j_is_{ij}x=s^i_jx$
\item$s^j_is^k_lx=s^k_ls^j_ix$
\item$s^j_is^k_ix=s^k_is^j_ix=s^j_is^k_jx$
\item$s^j_is^i_kx=s^j_ks_{ij}x$
\item$s^j_is^j_kx=s^j_kx$
\item$s^j_is^j_ix=s^j_ix$
\item$s^j_is^i_jx=s^j_ix$
\item$s^j_is_{ij}x=s_{ij}x$
\end{enumerate}
\end{definition}

Let $SA_n$ be the abstractly defined class $\mathbf{Mod}(\Sigma_n)$. In the above axiomatization, it is stipulated that 
$s_{ij}$ respects meet and join. From this it can be easily inferred that $s_{ij}$ respects $-$, so that it is in fact a Boolean endomorphism. 
Indeed if $x=-y$, then $x+y=1$ and $x.y=0$, hence $s_{ij}(x+y)=s_{ij}x+s_{ij}y=0$ and $s_{ij}(x.y)=s_{ij}x.s_{ij}y=0,$ hence
$s_{ij}x=-s_{ij}y$. We chose not to involve negation in our axiomtatization, to make it strictly positive. 

Note that different presentations of $^nn$ 
give rise to different axiomatizations, but of course they are all definitionally equivalent.
Here we are following the conventions of $\cite{HMT2}$ by distinguishing in notation between operations defined in abstract algebras,
and those defined in concrete set algebras. For example,
for $\A\in {\bf Mod}(\Sigma_n)$, $s_{ij}$ denotes the $i, j$ substitution operator, while in set algebras we denote
the (interpretation of this) operation by capital $S_{ij}$; similarly for $s_i^j$. This convention will be followed for all algebras considered in this
paper without any further notice.(Notice that the Boolean operations are also distinguished notationally).

To prove our main representation theorem, we need a few preparations: 

\begin{definition}
Let $R(U)=\{s_{ij}:i\neq j\in U\}\cup\{s^i_j:i\neq j\in U\}$ and let $\hat{}:R(U)^*\longrightarrow {}^UU$ be defined inductively as follows:
it maps the empty string to $Id_U$ and for any string $t$, $$(s_{ij}t)^{\hat{}}=[i,j]\circ t^{\hat{}}\;\;and\; (s^i_jt)^{\hat{}}=[i/j]\circ t^{\hat{}}.$$
\end{definition}

\begin{theorem}
For all $n\in\omega$ the set of (all instances of the) axiom-schemas 1 to 11 of Def.\ref{ax2}
is a presentation of the semigroup ${}^nn$ via generators $R(n)$ (see \cite{semigroup}).
That is, for all $t_1,t_2\in R(n)^*$ we have $$\mbox{1 to 11 of Def.\ref{ax2} }\vdash t_1=t_2\text{  iff  }t_1^{\hat{}}=t_2^{\hat{}}.$$
Here $\vdash$ denotes derivability using Birkhoff's calculus for equational logic.
\end{theorem}
\begin{proof}
This is clear because the mentioned schemas correspond exactly to the set of relations governing the generators
of ${}^nn$ (see \cite{semigroup}).
\end{proof}

\begin{definition}For every $\xi\in {}^nn$ we associate
a sequence $s_\xi\in R(U)^*$ (like we did before for $S_n$ using $^nn$ instead) such that $s_\xi^{\hat{}}=\xi.$ Such an $s_\xi$ exists, since $R(n)$ generates ${}^nn.$
\end{definition}

Like before, we have
\begin{lemma}\label{lemma}
Let $\A$ be an $RSA_n$ type $BAO$. Suppose $G\subseteq {}^nn$ is a locally square set,
and $\langle\mathcal{F}_\xi:\xi\in G\rangle$ is a system of ultrafilters of $\A$ such that for all $\xi\in G,\;i\neq j\in n$ and $a\in\A,$
the following conditions hold:
$${S_{ij}}^\A(a)\in\mathcal{F}_\xi\Leftrightarrow a\in \mathcal{F}_{\xi\circ[i,j]}\quad\quad (*),\text{and}$$ $${S^i_j}^\A(a)\in\mathcal{F}_\xi\Leftrightarrow a\in \mathcal{F}_{\xi\circ[i/j]}\quad\quad (**)$$  Then the following function $h:\A\longrightarrow\wp(G)$ is a homomorphism$$h(a)=\{\xi\in G:a\in \mathcal{F}_\xi\}.$$
\end{lemma}

Now, we show, unlike replacement algebras $RSA_n$ is a variety.
\begin{theorem}\label{variety} For any finite $n\geq 2$,
$RSA_n=SA_n$
\end{theorem}
\begin{proof}
Clearly, $RSA_n\subseteq SA_n$ because $SetSA_n\models\Sigma'_n$ (checking it is a routine computation).
Conversely, $RSA_n\supseteq SA_n$.
To see this, let $\A\in SA_n$ be arbitrary.
We may suppose that $\A$ has at least two elements, otherwise, it is easy to represent $\A$.
For every $0^\A\neq a\in A$ we will construct a homomorphism $h_a$ on $\A_{nn}$ such that $h_a(a)\neq 0^{\A_{nn}}$.
Let $0^\A\neq a\in A$ be an arbitrary element.
Let $\mathcal{F}$ be an ultrafilter over $\A$ containing $a$,
and for every $\xi\in {}^nn$, let $\mathcal{F}_\xi=\{z\in A: S^\A_\xi(z)\in\mathcal{F}\}$ (which is an ultrafilter).
(Here we use that all maps in $^nn$ are available, which we could not do before). 
Then, $h:\A\longrightarrow\A_{nn}$ defined by $h(z)=\{\xi\in{}^nn:z\in\mathcal{F}_{\xi}\}$ is a homomorphism by \ref{lemma} as $(*),$ $(**)$ hold.
\end{proof}
\begin{corollary} Simple algebras are finite.
\end{corollary}
\begin{proof} Finiteness follow from the previous theorem, since if $\A$ is simple, then the map $h$ defined above is injective. 
\end{proof}
However, not any finite algebra is simple.  Indeed if $V\subseteq {}^nn,$  and $s\in V$  is constant, then $\wp(\{s\})$ is a homomorphic image of 
$\wp(V)$. So if $|V|>2$, then this homomorphism will have a non-trivial kernel.
Let $Sir(SA_n)$ denote the class of subdirectly indecomposable algebrbra, and $Sim(SA_n)$ be the class of simple algebras.

\begin{question} Characterize the simple and subdirectly irreducible elements, is $SA_n$ a discriminator variety?
\end{question}
\begin{theorem}\label{super}$SA_n$ is locally finite, and has the superamalgamation property
\end{theorem}
\begin{proof}
For the first part let $\A\in SA_n$ be generated by $X$ and $|X|=m$. We claim that $|\A|\leq 2^{2^{m\times {}^nn}}.$
Let $Y=\{s_{\tau}x: x\in X, \tau\in {}^nn\}$.  Then
$\A=\Sg^{\Bl\A}Y$. This follows from the fact that the substitutions are Boolean endomorphisms.
Since $|Y|\leq m\times {}^nn,$ the conclusion follows.

For the second part, first a piece of notation. For an algebra $\A$ and $X\subseteq  A$, $fl^{\Bl\A}X$
denotes the Boolean filter generated by $X$. We show that the following  strong form of interpolation holds for the free algebras:
Let $X$ be a non-empty set. Let $\A=\Fr_XSA_n$, and let $X_1, X_2\subseteq \A$ be such that $X_1\cup X_2=X$.
Assume that $a\in \Sg^{\A}X_1$ and $c\in \Sg^{\A}X_2$ are such that $a\leq c$.
Then there exists an interpolant $b\in \Sg^{\A}(X_1\cap X_2)$ such that
$a\leq b\leq c$. Assume that $a\leq c$, but there is no such $b$. We will reach a contradiction.
Let $$H_1=fl^{\Bl\Sg^{\A}X_1}\{a\}=\{x: x\geq a\},$$
$$H_2=fl^{\Bl\Sg^{\A}X_2}\{-c\}=\{x: x\geq -c\},$$
and $$H=fl^{\Bl\Sg^{\A}(X_1\cap X_2)}[(H_1\cap \Sg^{\A}(X_1\cap X_2))\cup (H_2\cap \Sg^{\A}(X_1\cap X_2))].$$
We show that $H$ is a proper filter of $\Sg^{\A}(X_1\cap X_2)$.
For this, it suffices to show that for any $b_0,b_1\in \Sg^{\A}(X_1\cap X_2)$, for any $x_1\in H_1$ and $x_2\in H_2$
 if $a.x_1\leq b_0$ and $-c.x_2\leq b_1$, then $b_0.b_1\neq 0$.
Now $a.x_1=a$ and $-c.x_2=-c$. So assume, to the contrary, that $b_0.b_1=0$. Then $a\leq b_0$ and $-c\leq b_1$ and so
$a\leq b_0\leq-b_1\leq c$, which is impossible because we assumed that there is no interpolant.

Hence $H$ is a proper filter. Let $H^*$ be an ultrafilter of $\Sg^{\A}(X_1\cap X_2)$ containing $H$, and let $F$ be an ultrafilter of $\Sg^{\A}X_1$
and $G$ be an ultrafilter of $\Sg^{\A}X_2$ such that $$F\cap \Sg^{\A}(X_1\cap X_2))=H^*=G\cap \Sg^{\A}(X_1\cap x_2).$$
Such ultrafilters exist.


For simplicity of notation let $\A_1=\Sg^{\A}(X_1)$ and $\A_2=\Sg^{\A}(X_2).$
Define $h_1:\A_1\to \wp({}^nn)$ by
$$h_1(x)=\{\eta\in {}^nn: x\in s_{\eta}F\},$$
and
$h_2:\A_1\to \wp({}^nn)$ by
$$h_2(x)=\{\eta\in {}^nn: x\in s_{\eta}G_\},$$
Then $h_1, h_2$
are homomorphisms, they agree on $\Sg^{\A}(X_1\cap X_2).$
Indeed let $x\in \Sg^{\A}(X_1\cap X_2)$. Then $\eta\in h_1(x)$ iff $s_{\eta}x\in F$ iff
$s_{\eta}x\in F\cap \Sg^{\A}(X_1\cap X_2)=H^*=G\cap \Sg^{\A}(X_1\cap X_2)$ iff $s_{\eta}x\in G$ iff $\eta\in h_2(x)$.
Thus  $h_1\cup h_2$ is a function. By freeness
there is an $h:\A\to \wp({}^nn)$ extending $h_1$ and $h_2$. Now $Id\in h(a)\cap h(-c)\neq \emptyset$ which contradicts
$a\leq c$. The result now follows from \cite{Mak}, stating that the super amalgamtion property for a variety of $BAO$s follow from the interpolation property in the
free algebras.


\end{proof}

\section{Complete representability for $SA_n$}

For $SA_n$, the problem of complete representations is delicate since 
the substitutions corresponding to replacements may not be completey additive, 
and a complete representation, as we shall see, forces the complete additivity of the so represented
algebra.
In fact, as we discover, they are not.
We first show that representations may not preserve arbitrary joins, from which we infer that
the omitting types theorem fails, for the corresponding multi dimensional modal logic. 

Throughout this section $n$ is a natural number $\geq 2$.
All theorems in this subsection, with the exception of theorem \ref{additive}, apply to Pinter's algebras, 
by simply discarding the substitution operations corresponding to transpositions, and modifying the proofs accordingly.

\begin{theorem}\label{counter} There exists a countable $\A\in SA_n$ and $X\subseteq \A$, such that $\prod X=0$, 
but there is no representation $f:\A\to \wp(V)$
such that $\bigcap_{x\in X}f(x)=\emptyset$.
\end{theorem}
\begin{proof}
We give the example for $n=2$, and then we show how it extends to higher dimensions.

It suffices to show that there is an algebra $\A$, and a set $S\subseteq A$, such that $s_0^1$ does not preserves $\sum S$.
For if $\A$ had a representation as stated in the theorem, this would mean that $s_0^1$ is completely additive in $\A$.

For the latter statement, it clearly suffices to show that if $X\subseteq A$, and $\sum X=1$,
and there exists  an injection $f:\A\to \wp(V)$, such that $\bigcup_{x\in X}f(x)=V$,
then for any $\tau\in {}^nn$, we have $\sum s_{\tau}X=1$. So fix $\tau \in V$ and assume that this does not happen.
Then there is a $y\in \A$, $y<1$, and
$s_{\tau}x\leq y$ for all $x\in X$.
(Notice that  we are not imposing any conditions on cardinality of $\A$ in this part of the proof).
Now
$$1=s_{\tau}(\bigcup_{x\in X} f(x))=\bigcup_{x\in X} s_{\tau}f(x)=\bigcup_{x\in X} f(s_{\tau}x).$$
(Here we are using that $s_{\tau}$ distributes over union.)
Let $z\in X$, then $s_{\tau}z\leq y<1$, and so $f(s_{\tau}z)\leq f(y)<1$, since $f$ is injective, it cannot be the case that $f(y)=1$.
Hence, we have
$$1=\bigcup_{x\in X} f(s_{\tau}x)\leq f(y) <1$$
which is a contradiction, and we are done.
Now we turn to constructing the required  counterexample, which is an easy adaptation of a 
construction dut to Andr\'eka et all in \cite{AGMNS} to our present situation.
We give the detailed construction for the reader's conveniance.
\begin{enumarab}

\item Let $\B$ be an atomless Boolean set algebra with unit $U$, that has the following property:

For any distinct $u,v\in U$, there is $X\in B$ such that $u\in X$ and $v\in {}\sim X$.
For example $\B$ can be taken to be the Stone representation of some atomless Boolean algebra.
The cardinality of our constructed algebra will be the same as $|B|$.
Let $$R=\{X\times Y: X,Y\in \B\}$$
and
$$A=\{\bigcup S: S\subseteq R: |S|<\omega\}.$$
Then indeed we have $|R|=|A|=|B|$. 

We claim that $\A$ is a subalgebra of $\wp(^2U)$.

Closure under union is obvious. To check intersections, we have:
$$(X_1\times Y_1)\cap (X_2\times Y_2)=(X_1\cap X_2) \times (Y_1\cap Y_2).$$
Hence, if $S_1$ and $S_2$ are finite subsets of $R$, then
$$S_3=\{W\cap Z: W\in S_1, Z\in S_2\}$$
is also a finite subset of $R$ and we have
$$(\bigcup S_1)\cap (\bigcup S_2)=\bigcup S_3$$
For complementation:
$$\sim (X\times Y)=[\sim X\times U]\cup [U\times \sim Y].$$
If $S\subseteq R$ is finite, then
$$\sim \bigcup S=\bigcap \{\sim Z: Z\in S\}.$$
Since each $\sim Z$ is in $A$, and $A$ is closed under intersections, we conclude that
$\sim \bigcup S$ is in $A$.
We now show that it is closed under substitutions:
$$S_0^1(X\times Y)=(X\cap Y)\times U, \\ \ S_1^0(X\times Y)=U\times (X\cap Y)$$
$$S_{01}(X\times Y)=Y\times X.$$
Let
$$D_{01}=\{s\in U\times U: s_0=s_1\}.$$
We claim that the only subset of $D_{01}$ in $\A$ is the empty set.

Towards proving this claim, assume that $X\times Y$ is a non-empty subset of $D_{01}$.
Then for any $u\in X$ and $v\in Y$ we have $u=v$. Thus $X=Y=\{u\}$ for some $u\in U$.
But then $X$ and $Y$ cannot be in $\B$ since the latter is atomless and $X$ and $Y$ are atoms.
Let
$$S=\{X\times \sim X: X\in B\}.$$
Then
$$\bigcup S={}\sim D_{01}.$$
Now we show that
$$\sum{}^{\A}S=U\times U.$$
Suppose that $Z$ is an upper bound different from $U\times U$. Then $\bigcup S\subseteq Z.$ Hence
$\sim D_{01}\subseteq Z$, hence $\sim Z=\emptyset$, so $Z=U\times U$.
Now $$S_{0}^1(U\times U) =(U\cap U)\times U=U\times U.$$
But
$$S_0^1(X\times \sim X)=(X\cap \sim X)\times U=\emptyset.$$
for every $X\in B$.
Thus $$S_0^1(\sum S)=U\times U$$
and
$$\sum \{S_{0}^1(Z): Z\in S\}=\emptyset.$$

\item For $n>2$, one takes $R=\{X_1\times\ldots\times X_n: X_i\in \B\}$ and the definition of $\A$ is the same. Then,
in this case, one takes $S$ to be
$X\times \sim X\times U\times\ldots\times U$
such that $X\in B$. The proof survives verbatim.
\end{enumarab}
\end{proof}

By taking $\B$ to be countable, then $\A$ can be countable, and so it violates the omitting types theorem.

\begin{theorem}\label{additive}Let $\A$ be in $SA_n$. Then $\A$ is completely additive iff $s_0^1$ is completely additive, in particular if 
$\A$ is atomic, and $s_0^1$ is additive, then $\A$ is completely representable.
\end{theorem} 
\begin{proof} It suffices to show that for $i, j\in n$, $i\neq j$, we have $s_i^j$ is completely additive. 
But $[i|j]= [0|1]\circ [i,j]$, and $s_{[i,j]}$ is completely additive.
\end{proof}
For replacement algebras, when we do not have transpositions, so the above  proof does not work. So in principal we could have an algebra 
such that $s_0^1$ is completely additive in $\A$, while $s_1^0$ is 
not. 

\begin{question} Find a Pinter's algebra for which $s_0^1$ is completely additive but $s_1^0$ is not
\end{question}

However, like $SA_n$, we also have:

\begin{theorem} For every $n\geq 2$, and every distinct $i,j\in n$, there is an algebra $\B\in SA_n$ 
such that $s_i^j$ is not completely additive.
\end{theorem}
\begin{proof} One modifies the above example by letting $X$ occur in the $ith$ place of the product, 
and $\sim X$, in the $jth$ place.
\end{proof}
Now we turn to the notion of complete representability which is not remote from that of minimal completions \cite{complete}. 
\footnote{One way to show that varieties of representable algebras, like cylindric algebras, are not closed under completions, is to construct an 
atom structure $\F$, such that $\Cm\F$ is not representable, while $\Tm\F$, the subalgebra of 
$\Cm\F$, generated by the atoms,  is representable. This algebra cannot be completely representable;
because a complete representation 
induces a representation of the full complex algebra.}

\begin{definition} Let $\A\in SA_n$ and $b\in A$, then $\Rl_{b}\A=\{x\in \A: x\leq b\}$, with operations relativized to $b$.
\end{definition}

\begin{theorem}\label{atomic} Let $\A\in SA_n$  atomic  and assume that $\sum_{x\in X} s_{\tau}x=b$ for all $\tau\in {}^nn$.
Then $\Rl_{b}\A$ is completely representable.
In particular,  if $\sum_{x\in x}s_{\tau}x=1$, then $\A$ is completely representable.
\end{theorem}
\begin{proof} Clearly $\B=\Rl_{b}\A$ is atomic.
For all
$a\neq 0$, $a\leq b$, find an ultrafilter $F$ that contains $a$ and lies outside the nowhere dense set
$S\sim \bigcup_{x\in X} N_{s_{\tau}x}.$ This is possible since $\B$ is atomic, so one just takes the ultrafilter generated by an atom below $a$.
Define for each such $F$ and such $a$, $rep_{F,a}(x)=\{\tau \in {}^nn: s_{\tau}x\in F\}$; for each such $a$
let $V_a={}^nn$, and then set
$g:\B\to \prod_{a\ne 0} \wp(V_a)$ by $g(b)=(rep_{F,a}(b): a\leq b, a\neq 0)$.

\end{proof}

Since $SA_n$ can be axiomatized by Sahlqvist equations, it is closed under taking canonical extensions.
The next theorem says that canonical extensions have complete (square) representations. 
The argument used is borrowed from Hirsch and Hodkinson \cite{step} which is
a first order model-theoretic view of representability, using $\omega$-saturated models.
A model is $\omega$ saturated if every type that is finitely realized, is realized.
Every countable consistent theory has an $\omega$ saturated model.

\begin{theorem}\label{canonical} Let $\A\in SA_n$. Then $\A^+$ is completely representable on a square unit.
\end{theorem}
\begin{proof} For a given $\A\in SA_n$, we define a first order language $L(\A)$,
which is purely relational; it has one $n-ary$ relation symbol for each element of $\A$. Define
an $L(\A)$ theory $T(\A)$ as follows; for all $R,S,T\in \A$ and $\tau\in S_n$:
$$\sigma_{\lor}(R,S,T)=[R(\bar{x})\longleftrightarrow S(\bar{x})\lor T(\bar{x})], \\ R=S\lor T$$
$$\sigma_{\neg}(R,S)=(1(\bar{x})\to (R(\bar{x})\longleftrightarrow \neg S(\bar{x})], \\ R=\neg S$$
$$\sigma _{\tau}(R,S)=R(\bar{x})\longleftrightarrow s_{\tau}S(\bar{x}), \\R=s_{\tau}S.$$
$$\sigma_{\neq 0}(R)=\exists \bar{x}R(\bar{x}).$$
Let $\A$ be given. Then since $\A$ has a representation, hence $T(\A)$ is a consistent first order theory.
Let $M$ be an $\omega$ saturated model of $T({\A})$. Let $1^{M}={}^nM$.
Then for each $\bar{x}\in 1^{M}$, the set
$f_{\bar{x}}=\{a\in \A: M\models a(\bar{x})\}$ is an ultrafilter of $\A$.
Define
$h:\A^+\to \wp ({}^nM)$ via
$$S\mapsto \{\bar{x}\in 1^M: f_{\bar{x}}\in S\}.$$
Here $S$, an element of $\A^+$, is a set of ultrafilters of $\A$.
Clearly, $h(0)=h(\emptyset)=\emptyset$. $h$ respects complementation, for $\bar{x}\in 1^M$ and $S\in \A^+$,
$\bar{x}\notin h(S)$ iff $f_{\bar{x}}\notin S$
iff $f_{\bar{x}}\in -S$ iff $\bar{x}\notin h(-S).$
It is also straightforward to check that $h$ preserves arbitrary unions.
Indeed, we have $\bar{x}\in h(\bigcup S_i)$ iff $f_{\bar{x}}\in \bigcup S_i$ iff $\bar{x}\in h(S_i)$ for some $i$.

We now check that $h$ is injective. Here is where we use saturation. Let $\mu$ be an ultrafilter in $\A$,
we show that $h(\{\mu\})\neq \emptyset$. Take $p(\bar{x})=\{a(\bar{x}): a\in \mu\}$.
Then this type is finitely satisfiable in $M$. For if $\{a_0(\bar{x}),\ldots, a_{n-1}(\bar{x})\}$ is an arbitrary finite subset of
$p(\bar{x})$, then $a=a_0.a_1..a_{n-1}\in \mu$, so $a>0$. By axiom $\sigma_{\neq 0}(a)$ in $T(\A)$, we have
$M\models \exists\bar{x}a(\bar{x})$. Since $a\leq a_i$ for each $i<n$, we obtain using $\sigma_{\lor}(a_i, a, a_i)$ axiom of $T_{\A}$ that
$M\models \exists \bar{x}\bigwedge_{i<n}a_i(\bar{x})$, showing that $p(\bar{x})$ is finitely satisfiable as required.
Hence, by $\omega$ saturation $p$ is realized in $M$ by some $n$ tuple
$\bar{x}$. Hence $p$ is realized in $M$ by the tuple $\bar{x}$, say.
Now $\mu=f_{\bar{x}}$. So $\bar{x}\in h(\{\mu\}$ and we have proved that $h$ is an injection.
Preservation of the substitution operations is straightforward.

\end{proof}

\begin{lemma} For $\A\in SA_n$, the following two conditions are equivalent:
\begin{enumarab}
\item There exists a locally square set $V,$ and a complete representation $f:\A\to \wp(V)$.
\item For all non zero $a\in A$, there exists a homomorphism $f:\A\to \wp(^nn)$ such that $f(a)\neq 0$, and $\bigcup_{x\in \At\A} f(x)={}^nn$.
\end{enumarab}
\end{lemma}
\begin{demo}{Proof} Having dealt with the other implication before, we prove that (1) implies (2). 
Let there be given a complete representation $g:\A\to \wp(V)$.
Then $\wp(V)\subseteq \prod_{i\in I} \A_i$ for some set $I$, where $\A_i=\wp{}(^nn)$. Assume that $a$ is non-zero,
then $g(a)$ is non-zero, hence $g(a)_i$ is non-zero
for some $i$. Let $\pi_j$ be the $j$th projection $\pi_j:\prod \A_i\to \A_j$, $\pi_j[(a_i: i\in I)]=a_j$.
Define $f:\A\to \A_i$ by $f(x)=(\pi_i\circ g(x)).$
Then clearly $f$ is as required.
\end{demo}

The following theorem is a converse to \ref{atomic}

\begin{theorem}\label{converse} Assume that $\A\in SA_n$ is completely representable.
Let $f:\A\to \wp(^nn)$ be a non-zero homomorphism that is a complete representation. Then $\sum_{x\in X} s_{\tau}x=1$ for every $\tau\in {}^nn$.
\end{theorem}
\begin{proof} Like the first part of the proof of theorem \ref{counter}.
\end{proof}

Adapting another example in \cite{AGMNS} constructed for $2$ dimensional quasi-polyadic algebras, 
we show that atomicity and complete  representability
do not coincide for $SA_n$. Because we are lucky enough not have cylindrifiers, 
the construction works for all $n\geq 2$, 
and even for infinite dimensions, as we shall see. Here it is not a luxary to include the details, we have to do so
because we will generalize the example to all higher dimensions.

\begin{theorem}\label{counter2} There exists an atomic complete algebra in $SA_n$ ($2\leq n<\omega$), that is not completely representable.
Furthermore, dropping the condition of completeness, the algebra can be atomic and countable.
\end{theorem}
\begin{proof}
It is enough, in view of the previous theorem,  to construct an atomic algebra such that $\Sigma_{x\in \At\A} s_0^1x\neq 1$.
In what follows we produce such an algebra. (This algebra will be uncountable, due to the fact that it is infinite and complete, so it cannot be countable.
In particular, it cannot be used to violate the omitting types theorem, the most it can say is that the omitting types theorem fails 
for uncountable languages, which is not too much of a surprise).

Let $\mathbb{Z}^+$ denote the set of positive integers.
Let $U$ be an infinite set. Let $Q_n$, $n\in \omega$, be a family of $n$-ary relations that form  partition of $^nU$ 
such that $Q_0=D_{01}=\{s\in {}^nU: s_0=s_1\}$. And assume also that each $Q_n$ is symmetric; for any $i,j\in n$, $S_{ij}Q_n=Q_n$. 
For example one can $U=\omega$, and for $n\geq 1$, one sets 
$$Q_n=\{s\in {}^n\omega: s_0\neq s_1\text { and }\sum s_i=n\}.$$

Now fix $F$ a non-principal ultrafilter on $\mathcal{P}(\mathbb{Z}^+)$. For each $X\subseteq \mathbb{Z}^+$, define
\[
 R_X =
  \begin{cases}
   \bigcup \{Q_k: k\in X\} & \text { if }X\notin F, \\
   \bigcup \{Q_k: k\in X\cup \{0\}\}      &  \text { if } X\in F
  \end{cases}
\]

Let $$\A=\{R_X: X\subseteq \mathbb{Z}^+\}.$$
Notice that $\A$ is uncountable. Then $\A$ is an atomic set algebra with unit $R_{\mathbb{Z}^+}$, and its atoms are $R_{\{k\}}=Q_k$ for $k\in \mathbb{Z}^+$.
(Since $F$ is non-principal, so $\{k\}\notin F$ for every $k$).
We check that $\A$ is indeed closed under the operations.
Let $X, Y$ be subsets of $\mathbb{Z}^+$. If either $X$ or $Y$ is in $F$, then so is $X\cup Y$, because $F$ is a filter.
Hence
$$R_X\cup R_Y=\bigcup\{Q_k: k\in X\}\cup\bigcup \{Q_k: k\in Y\}\cup Q_0=R_{X\cup Y}$$
If neither $X$ nor $Y$ is in $F$, then $X\cup Y$ is not in $F$, because $F$ is an ultrafilter.
$$R_X\cup R_Y=\bigcup\{Q_k: k\in X\}\cup\bigcup \{Q_k: k\in Y\}=R_{X\cup Y}$$
Thus $A$ is closed under finite unions. Now suppose that $X$ is the complement of $Y$ in $\mathbb{Z}^+$.
Since $F$ is an ultrafilter exactly one of them, say $X$ is in $F$.
Hence,
$$\sim R_X=\sim{}\bigcup \{Q_k: k\in X\cup \{0\}\}=\bigcup\{Q_k: k\in Y\}=R_Y$$
so that  $\A$ is closed under complementation (w.r.t $R_{\mathbb{Z}^+}$).
We check substitutions. Transpositions are clear, so we check only replacements. It is not too hard to show that
\[
 S_0^1(R_X)=
  \begin{cases}
   \emptyset & \text { if }X\notin F, \\
   R_{\mathbb{Z}^+}      &  \text { if } X\in F
  \end{cases}
\]

Now
$$\sum \{S_0^1(R_{k}): k\in \mathbb{Z}^+\}=\emptyset.$$
and
$$S_0^1(R_{\mathbb{Z}^+})=R_{\mathbb{Z}^+}$$
$$\sum \{R_{\{k\}}: k\in \mathbb{Z}^+\}=R_{\mathbb{Z}^+}=\bigcup \{Q_k:k\in \mathbb{Z}^+\}.$$
Thus $$S_0^1(\sum\{R_{\{k\}}: k\in \mathbb{Z}^+\})\neq \sum \{S_0^1(R_{\{k\}}): k\in \mathbb{Z}^+\}.$$
For the completeness part, we refer to \cite{AGMNS}.

The countable algebra required is that generated by the countably many atoms.
\end{proof}

Our next theorem gives a plathora of algebras that are not completely representable. Any algebra which shares the atom structure of $\A$ 
constructed above cannot have a complete representation. Formally:

\begin{theorem} Let $\A$ be as in the previous example. Let $\B$ be an atomic an atomic algebra in $SA_n$ such that 
$\At\A\cong \At\B$. Then $\B$ is not completely representable
\end{theorem}
\begin{proof} Let such a $\B$ be given. Let $\psi:\At\A\to \At\B$ be an isomorphism of the atom structures (viewed as first order structures).
Assume for contradiction that $\B$ is completely representable, via $f$ say; $f:\B\to \wp(V)$ is an injective homomorphism such that
$\bigcup_{x\in \At\B}f(x)=V$. Define $g:\A\to \wp(V)$ by $g(a)=\bigcup_{x\in \At\A, x\leq a} f(\psi(x))$. Then, it can be easily checked that
$f$ establishes a complete representation of $\A$.
\end{proof} 
There is a wide spread belief, almost permenantly established that like cylindric algebras, any atomic {\it poyadic algebras of dimension $2$}
is completely representable. This is wrong.  The above example, indeed shows that it is not the case, 
because the set algebras consrtucted above , if we impose the additional condition that each $Q_n$ has $U$ as its domain and range,
 then the algebra in question becomes closed under the first two cylindrfiers, and by the same reasoning as above,
it {\it cannot} be completely representable.

However, this condition cannot be imposed for for higher dimension, and indeed for $n\geq 3$, the class
of completely representable quasiplyadic algebras is not elementary. When we have diagonal elements, 
the latter result holds for quasipolyadic equality algebras, but the former does not.
On the other hand, the variety of polyadic algebras of dimension $2$ is conjugated (which is not the case when we drop diagonal elements),
hence atomic representable algebras are completely representable.

We introduce a certain cardinal that plays an essential role in omitting types theorems \cite{Tarek}.

\begin{definition}

\begin{enumroman}

\item A Polish space is a complete separable metric space.

\item For a Polish space $X$, $K(X)$ denotes the 
ideal of meager subsets of $X$.
Set 
$$cov K(X)=min \{|C|: C\subseteq K(X),\  \cup C=X\}.$$
If $X$ is the real line,  or the Baire space $^{\omega}\omega$, or the 
Cantor set $^{\omega}2$, which are the prime examples of Polish spaces, 
we write $K$ instead of $K(X)$. 
\end{enumroman}
\end{definition}
The above three spaces are sometimes referred to as {\it real} spaces
since they are all Baire isomophic to the real line.
Clearly $\omega <covK \leq {}2^{\aleph_0}.$
The crdinal $covK$ is an important cardinal studied extensively in descriptive set theory, and it turns 
out strongly related to the number of types that can be omitted in consitent countable 
first order theory, a result of Newelski. It is known, but not trivial to show, that $covK$ is the least cardinal $\kappa$ 
such that the real space can be covered by $\kappa$ many closed nowhere dense 
sets,  that is the least cardinal such that the Baire Category Theorem fails. Also it is the largest cardinal 
for which Martin's axiom restricted to countable Boolean algebras
holds.

Indeed, the full Martin's axiom,
imply that $covK=2^{\aleph_0}$ but it is also consistent that $covK=\omega_1<2^{\aleph_0}.$  
Varying the value of $covK$ by (iterated) forcing in set theory is known. For example
$covK<2^{\aleph_0}$ is true in the random real model. It also holds in models 
constructed by forcings which do not add Cohen reals.

\begin{theorem} Let $\A\in SA_n$ be countable and completely additive, and let $\kappa$ be a cardinal $<covK$.
Assume that $(X_i: i<\kappa)$ is a family on non principal types. Then there exists a countable locally square set $V$, and
an injective homomorphism $f:\A\to \wp(V)$ such that $\bigcap_{x\in X_i} f(x)=\emptyset$  for each $i\in \kappa$.
 \end{theorem}
\begin{proof} 
Let $a\in A$ be non-zero. For each $\tau\in {}^nn$ for each $i\in \kappa$, let
$X_{i,\tau}=\{{\sf s}_{\tau}x: x\in X_i\}.$
Since the algebra is additive, then $(\forall\tau\in V)(\forall  i\in \kappa)\prod{}^{\A}X_{i,\tau}=0.$
Let $S$ be the Stone space of $\A$, and for $a\in N$, let $N_a$ be the clopen set consisting of all Boolean ultrafilters containing $a$.
Let $\bold H_{i,\tau}=\bigcap_{x\in X_i} N_{{\sf s}_{\tau}x}.$
Each $\bold H_{i,\tau}$ is closed and nowhere 
dense in $S$. Let $\bold H=\bigcup_{i\in \kappa}\bigcup_{\tau\in V}\bold H_{i,\tau}.$
By properties of $covK$, $\bold H$ is a countable collection of nowhere dense sets.
By the Baire Category theorem  for compact Hausdorff spaces, we get that $S\sim \bold H$ is dense in $S$.
Let $F$ be an ultrafilter that contains $a$ and is outside $\bold H$, that is $F$ intersects the basic set $N_a$; exists by density.  
Let $h_a(z)=\{\tau \in V: s_{\tau}\in F\}$, then $h_a$ is a homomorphism into $\wp(V)$ such that $h_a(a)\neq 0$.
Define $g:\A\to \prod_{a\in A}\wp(V)$ 
via $a\mapsto (h_a(x): x\in A)$. Let $V_{a}=V\times \{a\}$. 
Since $\prod_{a\in A}\wp(V)\cong \wp(\bigcup_{a\in A} V_a)$, 
then we are done for $g$ is clearly 
an injection.
\end{proof}
The statement of omitting $< {}^{\omega}2$ non-principal types is independent of $ZFC$.
Martins axiom offers solace here, in two respects. Algebras adressed could be uncountable 
(though still satisfying a smallness condition, that is a natural generalization of 
countability, and in fact, is necessary for Martin's axiom to hold), and types omitted can be $< {}^{\omega}2$.
More precisely:

\begin{theorem} In $ZFC+MA$ the following can be proved:
Let $\A\in SA_n$ be completely additive, and further assume that  $\A$ satisfies the countable chain condition 
(it contains no uncountable anti-chains). 
Let $\lambda<{}^{\omega}2$,  and let $(X_i: i<\lambda$) be a family of non-principal types. 
Then there exists a countable locally square set $V$, and
an injective homomorphism $f:\A\to \wp(V)$ such that $\bigcap_{x\in X_i} f(x)=\emptyset$  for each $i\in \lambda$.
\end{theorem}
\begin{proof}
The idea is exactly like the previous proof, except that now we have a union of $<{}^{\omega}2$ no where dense sets; 
the required ultrafilter to build the representation
we need lies outside this union. 
$MA$ offers solace here because it implies that such a union can be written as a countable union, and again the Baire category theorem readily 
applies.
\end{proof}
But without $MA$, if the given algebra is countable and completely additive, and if the types are maximal, then we can also omit $<{} ^{\omega}2$ types. 
This {\it is indeed } provable in $ZFC$, without any additional assumptions at all.
For brevity, when we have an omitting types theore, like the previos one,
we say that the hypothesis implies existence of a representation omitting the given set of non-principl types.
\begin{theorem} Let $\A\in SA_n$ be a countable, let $\lambda<{}^{\omega}2$ 
and $\bold F=(F_i: i<\lambda)$ be a family of non principal ultrafilters. 
Then there is a single representation that omits these non-principal ultrafilters.
\end{theorem}
\begin{proof} One can easily construct two representations that overlap only on principal ultrafilters \cite{h}. 
With a pinch of diagonalization this can be pushed to countable many.
But even more, using ideas of Shelah \cite{Shelah} thm 5.16, this this can be further pushed to $^{\omega}2$ 
many representations. Such representations are not necessarily pair-wise distinct, 
but they can be indexed by a set $I$ such that
$|I|={}^{\omega}2$, and if $i, j\in I$, are distinct and there is an ultrafilter that is realized in the representations indexed by $i$ and $j$, 
then it is principal. Note that they can be the same representation.
Now assume for contradiction that there 
is no representation  omitting the given non-principal ultrafilters. 

Then for all $i<{}^{\omega}2$,
there exists $F$ such that $F$ is realized in $\B_i$. Let $\psi:{}^{\omega}2\to \wp(\bold F)$, be defined by
$\psi(i)=\{F: F \text { is realized in  }\B_i\}$.  Then for all $i<{}^{\omega}2$, $\psi(i)\neq \emptyset$.
Furthermore, for $i\neq j$, $\psi(i)\cap \psi(j)=\emptyset,$ for if $F\in \psi(i)\cap \psi(j)$ then it will be realized in
$\B_i$ and $\B_j$, and so it will be principal.
This implies that $|\F|={}^{\omega}2$ which is impossible.

\end{proof}

In case of omitting the single special type, $\{-x: x\in \At\A\}$ for an atomic algebra, no conditions whatsoever  on cardinalities are
pre-supposed.

\begin{theorem} If $\A\in SA_n$ is completelyadditive and atomic, then $\A$ is completely representable
\end{theorem}
\begin{proof}
Let $\B$ be an atomic transposition algebra, let $X$ be the set of atoms, and
let $c\in \B$ be non-zero. Let $S$ be the Stone space of $\B$, whose underlying set consists of all Boolean ulltrafilters of
$\B$. Let $X^*$ be the set of principal ultrafilters of $\B$ (those generated by the atoms).
These are isolated points in the Stone topology, and they form a dense set in the Stone topology since $\B$ is atomic.
So we have $X^*\cap T=\emptyset$ for every nowhere dense set $T$ (since principal ultrafilters, which are isolated points in the Stone topology,
lie outside nowhere dense sets).
Recall that for $a\in \B$, $N_a$ denotes the set of all Boolean ultrafilters containing $a$.

Now  for all $\tau\in S_n$, we have
$G_{X, \tau}=S\sim \bigcup_{x\in X}N_{s_{\tau}x}$
is nowhere dense. Let $F$ be a principal ultrafilter of $S$ containing $c$.
This is possible since $\B$ is atomic, so there is an atom $x$ below $c$; just take the
ultrafilter generated by $x$. Also $F$ lies outside the $G_{X,\tau}$'s, for all $\tau\in S_n.$
Define, as we did before,  $f_c$ by $f_c(b)=\{\tau\in S_n: s_{\tau}b\in F\}$.
Then clearly for every $\tau\in S_n$ there exists an atom $x$ such that $\tau\in f_c(x)$, so that $S_n=\bigcup_{x\in \At\A} f_c(x)$

For each $a\in A$, let
$V_a=S_n$ and let $V$ be the disjoint union of the $V_a$'s.
Then $\prod_{a\in A} \wp(V_a)\cong \wp(V)$. Define $f:\A\to \wp(V)$ by $f(x)=g[(f_ax: a\in A)]$.
Then $f: \A\to \wp(V)$ is an embedding such that
$\bigcup_{x\in \At\A}f(x)=V$. Hence $f$ is a complete representation.

\end{proof}
Another proof inspired from modal logic, and taken from Hirsch and Hodkinson \cite{atomic}, with the slight difference that we assume complete additivity not 
conjugation. Let $\A\in SA_n$ be additive and atomic, so the first order correspondants of 
the equations are valid in its atom structure $\At\A$. 
But $\At\A$ is a bounded image of a disjoint union of square frames $\F_i$, that is there exists a bounded morphism from $\At\A\to \bigcup_{i\in I}\F_i$, 
so that $\Cm\F_i$ is a full set algebra.  Dually the inverse of this bounded morphism is an embedding frm $\A$ into $\prod_{i\in I}\Cm\F_i$ that
preserves all meets. 

The previous theorem  tells us how to capture (in first order logic) complete representability. 
We just spell out first order axioms forcing complete additivity of substitutions corresponding to replacements. 
Other substitutions, corresponding to the transpositions, are easily seen to be completely additive anway; in fact, they are 
self-conjugate.

\begin{theorem}\label{elementary} The class $CSA_n$ is elementary and is finitely axiomatizable in first order logic.
\end{theorem}
\begin{proof} We assume $n>1$, the other cases degenerate to the Boolean case.
Let $\At(x)$ be the first order formula expressing that $x$ is an atom. That is $\At(x)$ is the formula
$x\neq 0\land (\forall y)(y\leq x\to y=0\lor y=x)$. For distinct $i,j<n$ let $\psi_{i,j}$ be the formula:
$y\neq 0\to \exists x(\At(x)\land s_i^jx\neq 0\land s_i^jx\leq y).$ Let $\Sigma$ be obtained from $\Sigma_n$ by adding $\psi_{i,j}$
for every distinct $i,j\in n$.

We show that $CSA_n={\bf Mod}(\Sigma)$. Let $\A\in CSA_n$. Then, by theorem \ref{converse}, we have
$\sum_{x\in X} s_i^jx=1$ for all $i,j\in n$. Let $i,j\in $ be distinct. Let $a$ be non-zero, then $a.\sum_{x \in X}s_i^jx=a\neq 0$,
hence there exists $x\in X$, such that
$a.s_i^jx\neq 0$, and  so $\A\models \psi_{i,j}$.
Conversely, let $\A\models \Sigma$. Then for all $i,j\in n$, $\sum_{x\in X} s_i^jx=1$. Indeed, assume that for some distinct
$i,j\in n$, $\sum_{x\in X}s_i^jx\neq 1$.
Let $a=1-\sum_{x\in X} s_i^jx$. Then $a\neq 0$. But then there exists $x\in X$, such that $s_i^jx.\sum_{x\in X}s_i^jx\neq 0$
which is impossible. But for distinct $i, j\in n$, we have  $\sum_{x\in X}s_{[i,j]}X=1$ anyway, and so $\sum s_{\tau}X=1$, for all
$\tau\in {}^nn$, and so it readily follows that $\A\in CRA_n.$
\end{proof}

\begin{definition} Call a completely representable algebra {\it square completely representable}, 
if it has a complete representation on a square.
\end{definition}
\begin{theorem} If $\A\in SA_n$ is completely representable, then $\A$ is square completely representable. 
\end{theorem}
\begin{proof} Assume that $\A$ is completely representable. Then each $s_i^j$ is completely additive for all $i,j\in n$. 
For each $a\neq 0$, choose $F_a$ outside the nowhere dense sets $S\sim \bigcup_{x\in \At \A}N_{s_{\tau}x}$, $\tau\in {}^nn$,
and define $h_a:\A\to \wp(^nn)$ the usual way, that is $h(x)=\{\tau\in {}^nn: s_{\tau}a\in F_a\}.$ 
Let $V_a={}^nn\times \{a\}.$  Then $h:\A\to \prod_{a\in \A}\wp(V_a)$ defined via
$a\mapsto (h_a(x): a\in \A)$ is a complete representation.
But $\prod_{a\in \A, a\neq 0}\wp(V_a)\cong \wp(\bigcup_{a\in \A, a\neq 0} V_a)$ and the latter is square.
\end{proof} 
A variety $V$ is called completion closed if whenever $\A\in V$ is completely additive then its minimal completion $\A^*$ (which exists) is in $V$.
On completions, we have:
\begin{theorem} If $\A\in SA_n$ is completely additive, then $\A$ has a strong completion $\A^*$. 
Furthermore, $\A^*\in SA_n.$ In other words, $SA_n$ is completion closed.
\end{theorem}
\begin{proof} The completion can be constructed because the algebra is completely additive. 
The second part follows from the fact that the stipulaed positive equations axiomatizing $SA_n$ 
are preserved in completions \cite{v3}.
\end{proof}
We could add diagonal elements $d_{ij}'s$ to our signature and consider $SA_n$ enriched by diagonal elements, call this class $DSA_n$.
In set algebras with unit $V$ a locally square unit, the diagonal $d_{ij}, i,j \in n$ will be interpreted as $D_{ij}=\{s\in V: s_i=s_j\}$.
All positive results, with no single exception, established for the diagonal fre case, i.e. for $SA_n$  
will extend to $DSA_n$, as the reader is invited to show. 
However,
the counterexamples constructed above to violate complete representability of atomic algebras, and an omitting types theorem for the corresponding 
muti-dimensional modal logic {\it do not work} in this new context, because such algebras
do not contain the diagonal $D_{01}$, and this part is essential in the proof. We can show though 
that again the class of completely represenable algebras is elementary.

We will return to such an issue in the infinite dimensional csae, where even more interesting results hold; for example square representaion and weak ones
form an interesting dichotomy. 

\subsection{The Infinite Dimensional Case}

For $SA$'s, we can lift our results to infinite dimensions.

We give, what we believe is  a reasonable generalization to the above theorems for the infinite dimensional case, by allowing weak sets as units,
a weak set being a set of sequences that agree
cofinitely with some fixed sequence.

That is a weak set is one of the form $\{s\in {}^{\alpha}U: |\{i\in \alpha, s_i\neq p_i\}|<\omega\}$,
where $U$ is a set, $\alpha$ an ordinal and $p\in {}^{\alpha}U$.
This set will be denoted by $^{\alpha}U^{(p)}$. The set $U$ is called the base of the weak set.
A set $V\subseteq {}^{\alpha}\alpha^{(Id)}$, is defined to
be  di-permutable just like the finite dimensional case. Altering top elements to be weak sets, rather than squares,
turns out to be   fruitful and rewarding.

\begin{definition}
We let  $WSA_{\alpha}$ be the variety generated by
$$\wp(V)=\langle\mathcal{P}(V),\cap,-,S_{ij}, S_i^j\rangle_{i,j\in\alpha},$$
where $V\subseteq {}^\alpha\alpha^{(Id)}$ is locally square 
(Recall that $V$ is locally square, if whenever $s\in V$, then, $s\circ [i|j]$ and  $s\circ [i,j]\in V$, for $i,j\in \alpha$.)

\end{definition}
Let $\Sigma_{\alpha}$ be the set of finite schemas obtained from
$\Sigma_n$ but now allowing indices from $\alpha$.

Obviously $\Sigma_{\alpha}$ is infinite, but it has a finitary feature in a two sorted sense.
One sort for the ordinals $<\alpha$, the other for the first order situation.

Indeed, the system $({\bf Mod}(\Sigma_{\alpha}): \alpha\geq \omega)$ is an instance of
what is known in the literature as a system of varieties definable by a finite schema, which means that it is enough to specify a strictly
finite subset of $\Sigma_{\omega}$,
to define $\Sigma_{\alpha}$ for all $\alpha\geq \omega$. (Strictly speaking, systems of varieties definable by schemes require
that we have a unary operation behaving like a cylindrifier, but such a distinction is immaterial in the present context.)


More precisely, let $L_{\kappa}$ be the language of $WSA_{\kappa}$.
Let $\rho:\alpha\to \beta$ be an injection. One defines for a term $t$ in $L_{\alpha}$ a term $\rho(t)$ in $L_{\beta}$ by recursion
as follows: For any variable $v_i$, $\rho(v_i)=v_i$ and for any unary operation $f$
$\rho(f(\tau))=f(\rho(\tau))$.
Now let $e$ be a given equation in the language $L_{\alpha}$, say $e$ is the equation $\sigma=\tau$.
then one defines $\rho(e)$ to be the equation $\rho(\sigma)=\rho(\tau)$.

Then there exists a finite set $\Sigma\subseteq \Sigma_{\omega}$ such that
$\Sigma_{\alpha}=\{\rho(e): \rho:\omega\to \alpha \text { is an injection },e\in \Sigma\}.$
Notice that $\Sigma_{\omega}=\bigcup_{n\geq 2} \Sigma_n.$

Let $SA_{\alpha}={\bf Mod}(\Sigma_{\alpha})$. We give two proofs of the next main representation theorem, but first a definition.

\begin{definition} Let $\alpha\leq\beta$ be ordinals and let $\rho:\alpha\rightarrow\beta$ be an injection.
For any $\beta$-dimensional algebra $\B$
we define an $\alpha$-dimensional algebra $\Rd^\rho\B$, with the same base and Boolean structure as
$\B$, where the $(i,j)$th transposition substitution  of $\Rd^\rho\B$ is $s_ {\rho(i)\rho(j)}\in\B$, and same for replacements.

For a class $K$, $\Rd^{\rho}K=\{\Rd^{\rho}\A: \A\in K\}$. When $\alpha\subseteq \beta$ and $\rho$ is the identity map on $\alpha$, then we write
$\Rd_{\alpha}\B$, for $\Rd^{\rho}\B$.
\end{definition}
Our first proof, is more general than the present context; it is basically a lifting argument that can be used to transfer results in the finite dimensional case to 
infinite dimensions.
\begin{theorem}\label{r} For any infinite ordinal $\alpha$, $SA_{\alpha}=WSA_{\alpha}.$
\end{theorem}
\begin{demo}{First proof}
\begin{enumarab}
\item First for $\A\models \Sigma_{\alpha}$, $\rho:n\to \alpha,$ an injection, and $n\in \omega,$ we have  $\Rd^{\rho}\A\in SA_n$.

\item For any $n\geq 2$ and $\rho:n\to \alpha$ as above, $SA_n\subseteq\mathbf{S}\Rd^{\rho}WSA_{\alpha}$ as in \cite{HMT2} theorem 3.1.121.

\item $WTA_{\alpha}$ is, by definition,  closed under ultraproducts.

\end{enumarab}

Now we show that if $\A\models \Sigma_{\alpha}$, then $\A$ is representable.
First, for any $\rho:n\to \alpha$, $\Rd^{\rho}\A\in SA_n$. Hence it is in $S\Rd^{\rho}WSA_{\alpha}$.
Let $I$ be the set of all finite one to one sequences with range in $\alpha$.
For $\rho\in I$, let $M_{\rho}=\{\sigma\in I:\rho\subseteq \sigma\}$.
Let $U$ be an ultrafilter of $I$ such that $M_{\rho}\in U$ for every $\rho\in I$. Exists, since $M_{\rho}\cap M_{\sigma}=M_{\rho\cup \sigma}.$
Then for $\rho\in I$, there is $\B_{\rho}\in WSA_{\alpha}$ such that
$\Rd^{\rho}\A\subseteq \Rd^{\rho}\B_{\rho}$. Let $\C=\prod\B_{\rho}/U$; it is in $\mathbf{Up}WSA_{\alpha}=WSA_{\alpha}$.
Define $f:\A\to \prod\B_{\rho}$ by $f(a)_{\rho}=a$ , and finally define $g:\A\to \C$ by $g(a)=f(a)/U$.
Then $g$ is an embedding.
\end{demo}
The \textbf{second proof} follows from the next lemma,
whose proof is identical to the finite dimensional case with obvious modifications.
Here, for $\xi\in {}^\alpha\alpha^{(Id)},$ the operator $S_\xi$ works as $S_{\xi\upharpoonright J}$ (which can be defined as in in the finite dimensional case because we have a 
finite support)
where $J=\{i\in\alpha:\xi(i)\neq i\}$ (in case $J$ is empty, i.e., $\xi=Id_\alpha,$ $S_\xi$ is the identity operator).
\begin{lemma}\label{f}
Let $\A$ be an $SA_\alpha$ type $BAO$ and $G\subseteq{}^\alpha\alpha^{(Id)}$ permutable.
Let $\langle\mathcal{F}_\xi:\xi\in G\rangle$ is a system of ultrafilters of $\A$
such that for all $\xi\in G,\;i\neq j\in \alpha$ and $a\in\A$
the following condition holds:$$S_{ij}^\A(a)\in\mathcal{F}_\xi\Leftrightarrow a\in \mathcal{F}_{\xi\circ[i,j]}\quad\quad (*).$$
Then the following function $h:\A\longrightarrow\wp(G)$
is a homomorphism
$$h(a)=\{\xi\in G: a\in \mathcal{F}_\xi\}.$$
\end{lemma}


\begin{definition}
We let $WSA_{\alpha}$ be the variety generated by
$$\wp(V)=\langle\mathcal{P}(V),\cap,\sim,S^i_j,S_{ij}\rangle_{i,j\in\alpha}, \ \ V\subseteq{}^\alpha\alpha^{(Id)}$$ is locally square.

\end{definition}
Let $\Sigma_{\alpha}$ be the set of finite schemas obtained from
from the $\Sigma_n$ but now allowing indices from $\alpha$; and let
$SA_{\alpha}={\bf Mod}(\Sigma_{\alpha}')$.  Then as before, we can prove, completeness and interpolation 
(for the corresponding multi dimensional modal logic):

\begin{theorem} \label{1}\begin{enumarab} Let $\alpha$ be an infinite ordinal. Then, we have:
\item $WSA_{\alpha}=SA_{\alpha}$
\item $SA_{\alpha}$ has the superamalgamation property
\end{enumarab}
\end{theorem}
\begin{proof}  Like before undergoing the obvious modifications. 
\end{proof}

In particular, from the first item, it readily follows, that if $\A\subseteq \wp(^{\alpha}U)$ and $a\in A$ is non-zero,
then there exists a homomorphism $f:\A\to \wp(V)$ for some permutable $V$ such that $f(a)\neq 0$.
We shall  prove a somewhat deep converse of this result, that will later enable us to verify that the quasi-variety of 
subdirect products of full set algebras is a variety.

But first a definition and a result on the number of non-isomorphic models.
\begin{definition}
Let $\A$ and $\B$ be set algebras with bases $U$ and $W$ respectively. Then $\A$ and $\B$
are \emph{base isomorphic} if there exists a bijection
$f:U\to W$ such that $\bar{f}:\A\to \B$ defined by ${\bar f}(X)=\{y\in {}^{\alpha}W: f^{-1}\circ y\in x\}$ is an isomorphism from $\A$ to $\B$
\end{definition}
\begin{definition} An algebra $\A$ is \emph{hereditary atomic}, if each of its subalgebras is atomic.
\end{definition}
Finite Boolean algebras are hereditary atomic of course,
but there are infinite hereditary atomic Boolean algebras, example any Boolean algebra generated 
by the atoms is necessarily hereditory atomic, like the finite cofinite Boolean algebra. 
An algebra that is infinite and complete, like that in our example violating complete representability,
is not hereditory atomic, whether atomic or not.

\begin{example}
Hereditary atomic algebras arise naturally as the Tarski Lindenbaum algebras of
certain countable first order theories, that abound. If $T$ is a countable complete first order theory
which has an an $\omega$-saturated model, then for each $n\in \omega$,
the Tarski Lindenbuam Boolean algebra $\Fm_n/T$ is hereditary atomic. Here $\Fm_n$ is the set of formulas using only
$n$ variables. For example $Th(\Q,<)$ is such with $\Q$ the $\omega$ saturated model.
\end{example}

A well known model-theoretic result is that $T$ has an $\omega$ saturated model iff $T$ has countably many $n$ types
for all $n$. Algebraically $n$ types are just ultrafilters in $\Fm_n/T$.
And indeed, what characterizes hereditary atomic algebras is that the base of their Stone space, that is the set of all
ultrafilters, is at most countable.

\begin{lemma}\label{b} Let $\B$ be a countable  Boolean algebra. If $\B$ is hereditary atomic then 
the number of ultrafilters is at most countable; of course they are finite
if $\B$ is finite. If $\B$ is not hereditary atomic then it has exactly $2^{\omega}$ ultrafilters.
\end{lemma}
\begin{proof}\cite{HMT1} p. 364-365  for a detailed discussion.
\end{proof}

A famous conjecture of Vaught says that the number of non-isomorphic countable models of a complete theory
is either $\leq \omega$ or exactly $^{\omega}2$. We show that this is the case for the multi (infinite) dimensional modal logic corresponding
to $SA_{\alpha}$. Morleys famous theorem excluded all possible cardinals in between except for $\omega_1$.

\begin{theorem}\label{2} Let $\A\in SA_{\omega}$ be countable and simple.
Then the number of non base isomorphic representations of $\A$ is either $\leq \omega$ or $^{\omega}2$. Furthermore, 
if $\A$ is assumed completely additive, and $(X_i: i< covK)$ is a family of non-principal types, then the number of models omitting 
these types is the same.
\end{theorem}
\begin{proof} For the first part. If $\A$ is hereditary atomic, then the number of models $\leq$ the number of ultrafilters, hence is at most countable.
Else, $\A$ is not hereditary atomic, then it has $^{\omega}2$ ultrafilters. For an ultrafilter $F$, let $h_F(a)=\{\tau \in V: s_{\tau}a\in F\}$, $a\in \A$.
Then $h_F\neq 0$, indeed $Id\in h_F(a)$ for any $a\in \A$, hence $h_F$ is an injection, by simplicity of $\A$.
Now $h_F:\A\to \wp(V)$; all the $h_F$'s have the same target algebra.
We claim that $h_F(\A)$ is base isomorphic to $h_G(\A)$ iff there exists a finite bijection $\sigma\in V$ such that
$s_{\sigma}F=G$.
We set out to confirm our claim. Let $\sigma:\alpha\to \alpha$ be a finite bijection such that $s_{\sigma}F=G$.
Define $\Psi:h_F(\A)\to \wp(V)$ by $\Psi(X)=\{\tau\in V:\sigma^{-1}\circ \tau\in X\}$. Then, by definition, $\Psi$ is a base isomorphism.
We show that $\Psi(h_F(a))=h_G(a)$ for all $a\in \A$. Let $a\in A$. Let $X=\{\tau\in V: s_{\tau}a\in F\}$.
Let $Z=\Psi(X).$ Then
\begin{equation*}
\begin{split}
&Z=\{\tau\in V: \sigma^{-1}\circ \tau\in X\}\\
&=\{\tau\in V: s_{\sigma^{-1}\circ \tau}(a)\in F\}\\
&=\{\tau\in V: s_{\tau}a\in s_{\sigma}F\}\\
&=\{\tau\in V: s_{\tau}a\in G\}.\\
&=h_G(a)\\
\end{split}
\end{equation*}
Conversely, assume that $\bar{\sigma}$ establishes a base isomorphism between $h_F(\A)$ and $h_G(\A)$.
Then $\bar{\sigma}\circ h_F=h_G$.  We show that if $a\in F$, then $s_{\sigma}a\in G$.
Let $a\in F$, and let $X=h_{F}(a)$.
Then, we have
\begin{equation*}
\begin{split}
&\bar{\sigma\circ h_{F}}(a)=\sigma(X)\\
&=\{y\in V: \sigma^{-1}\circ y\in h_{F}(X)\}\\
&=\{y\in V: s_{\sigma^{-1}\circ y}a\in F\}\\
&=h_G(a)\\
\end{split}
\end{equation*}
Now we have $h_G(a)=\{y\in V: s_{y}a\in G\}.$
But $a\in F$. Hence $\sigma^{-1}\in h_G(a)$ so $s_{\sigma^{-1}}a\in G$, and hence $a\in s_{\sigma}G$.

Define the equivalence relation $\sim $ on the set of ultrafilters by $F\sim G$, if there exists a finite permutation $\sigma$
such that $F=s_{\sigma}G$. Then any equivalence class is countable, and so we have $^{\omega}2$ many orbits, which correspond to
the non base isomorphic representations of $\A$.

For the second part,  suppose we want to count the number of representations omitting a
family $\bold X=\{X_i:i<\lambda\}$ ($\lambda<covK)$
of non-isolated types of $T$. We assume, without any loss of generality, that the dimension is $\omega$.
Let $X$ be stone space of $\A$. 
Then $$\mathbb{H}={\bf X}\sim  \bigcup_{i\in\lambda,\tau\in W}\bigcap_{a\in X_i}N_{s_\tau a}$$
(where $W=\{\tau\in{}^\omega\omega:|\{i\in \omega:\tau(i)\neq i\}|<\omega\}$) is clearly (by the above discussion)
the space of ultrafilters corresponding to representations omitting $\Gamma.$ Note that $\mathbb{H}$ the intersection of two dense sets.

But then by properties of $covK$ the union  $\bigcup_{i\in\lambda}$
can be reduced to a countable union.
We then have $\mathbb{H}$ a $G_\delta$ subset of a Polish space, namely the Stone space $X$.
So $\mathbb{H}$ is
Polish and moreover, $\mathcal{E}'=\sim \cap (\mathbb{H}\times \mathbb{H})$
is a Borel equivalence relation on $\mathbb{H}.$ 
It follows then that the number of representations  omitting $\Gamma$
is either countable or else $^{\omega}2.$

\end{proof}
The above theorem is not so  deep, as it might appear on first reading. The relatively simple 
proof is an instance of the obvious fact that if a countable Polish group, acts on an uncountable Polish space, then the number of induced orbits
has the cardinality of the continuum, because it factors out an uncountable set by a countable one.  
When the Polish group is uncountable, finding the number of orbits is still an open question, of which Vaught's
conjecture is an instance (when the group is the symmetric group on $\omega$ actong on the Polish space of pairwise non-isomorphic models.)

We shall prove that weak set algebras are strongly isomorphic to set algebras in the sense of the following definition. 
This will enable us to show that $RSA_{\alpha}$, like the finite dimensional case, is also a variety.

\begin{definition}
Let $\A$ and $\B$ be set algebras with units $V_0$ and $V_0$ and bases $U_0$ and $U_1,$ respectively,
and let $F$ be an isomorphism from $\B$ to $\A$.
Then $F$ is a \emph{strong ext-isomorphism} if $F=(X\cap V_0: X\in B)$. In this case $F^{-1}$ is called a \emph{strong subisomorphism}. 
An isomorphism
$F$ from $\A$ to $\B$ is a \emph{strong ext base isomorphism} if $F=g\circ h$
for some base isomorphism and some strong ext isomorphism $g$. In this case $F^{-1}$ is called a {\it strong sub base isomorphism.}

\end{definition}

The following, this time deep theorem, uses ideas of Andr\'eka and N\'emeti, reported in \cite{HMT2}, theorem 3.1.103, 
in how to square units of so called weak cylindric set algebras (cylindric algebras whose units are weak spaces):

\begin{theorem}\label{weak} If $\B$ is a subalgebra of $ \wp(^{\alpha}\alpha^{(Id)})$ then there exists a set algebra $\C$ with unit $^{\alpha}U$
such that $\B\cong \C$. Furthermore, the isomorphism is a strong sub-base isomorphism.
\end{theorem}

\begin{proof}We square the unit using ultraproducts.
We prove the theorem for $\alpha=\omega$. We warn the reader that the proof uses heavy machinery of proprties of ultraproducts for algebras consisting of
infinitary relations. Let $F$ be a non-principal ultrafilter
over $\omega$. (For $\alpha>\omega$, one takes an $|{\alpha}^+|$ regular ultrafilter on $\alpha$).
Then there exists a function
$h: \omega\to \{\Gamma\subseteq_{\omega} \omega\}$
such that $\{i\in \omega: \kappa\in h(i)\}\in F$ for all $\kappa<\omega$.
Let $M={}^{\omega}U/F$.  $M$ will be the base of our desired algebra, that is  $\C$ will
have unit $^{\omega}M.$
Define $\epsilon: U\to {}^{\omega}U/F$ by
$$\epsilon(u)=\langle u: i\in \omega\rangle/F.$$
Then it is clear that $\epsilon$ is one to one.
For $Y\subseteq {}^{\omega}U$,
let $$\bar{\epsilon}(Y)=\{y\in {}^{\omega}(^{\omega}U/F): \epsilon^{-1}\circ y\in Y\}.$$
By an $(F, (U:i\in \omega), \omega)$ choice function we mean a function
$c$ mapping $\omega\times {}^{\omega}U/F$
into $^{\omega}U$ such that for all $\kappa<\omega$
and all $y\in {}^{\omega}U/F$, we have $c(k,y)\in y.$
Let $c$ be an $(F, (U:i\in \omega), \omega)$
choice function satisfying the following condition:
For all $\kappa, i<\omega$ for all $y\in X$, if
$\kappa\notin h(i)$ then $c(\kappa,y)_i=\kappa$,
if $\kappa\in h(i)$ and $y=\epsilon u$ with  $u\in U$ then $c(\kappa,y)_i=u$.
Let $\delta: \B\to {}^{\omega}\B/F$ be the following monomorphism
$$\delta(b)=\langle b: i\in \omega\rangle/F.$$
Let $t$ be the unique homomorphism
mapping
$^{\omega}\B/F$ into $\wp{}^{\omega}(^{\omega}U/F)$
such that  for any $a\in {}^{\omega}B$
$$t(a/F)=\{q\in {}^{\omega}(^{\omega}U/F): \{i\in \omega: (c^+q)_i\in a_i\}\in F\}.$$
Here $(c^+q)_i=\langle c(\kappa,q_\kappa)_i: k<\omega\rangle.$
It is easy to show that show that $t$ is well-defined. Assume that $J=\{i\in \omega: a_i=b_i\}\in F$. If $\{i\in \omega: (c^+q)_i\in a_i\}\in F$,
then $\{i\in \omega; (c^+q)_i\in b_i\}\in F$. The converse inclusion is the same, and we are done.

Now we check that the map preserves the operations. That the  Boolean operations are preserved is obvious.

So let us check substitutions. It is enough to consider transpositions and replacements.
Let $i,j\in \omega.$  Then $s_{[i,j]}g(a)=g(s_{[i,j]}a)$,
follows from the simple observation that $(c^+q\circ [i,j])_k\in a$ iff $(c^+q)_k\in s_{[i,j]}a$.
The case of replacements is the same;  $(c^+q\circ [i|j])_k\in a$ iff $(c^+q)_k\in s_{[i|j]}a.$

Let $g=t\circ \delta$. Then for $a\in B$, we have
$$g(a)=\{q\in {}^{\omega}(^{\omega}U/F): \{i\in \omega: (c^+q)_i\in a\}\in F\}.$$
Let $\C=g(\B)$. Then $g:\B\to \C$.
We show that $g$ is an isomorphism
onto a set algebra. First it is clear that $g$ is a monomorphism. Indeed if $a\neq 0$, then $g(a)\neq \emptyset$.
Now $g$ maps $\B$ into an algebra with unit $g(V)$.

Recall that $M={}^{\omega}U/F$. Evidently $g(V)\subseteq {}^{\omega}M$.
We show the other inclusion. Let $q\in {}^{\omega}M$. It suffices to show that
$(c^+q)_i\in V$ for all $i\in\omega$. So, let $i\in \omega$. Note that
$(c^+q)_i\in {}^{\omega}U$. If $\kappa\notin h(i)$ then we have
$$(c^+q)_i\kappa=c(\kappa, q\kappa)_i=\kappa.$$
Since $h(i)$ is finite the conclusion follows.
We now prove that for $a\in B$
$$(*) \ \ \ g(a)\cap \bar{\epsilon}V=\{\epsilon\circ s: s\in a\}.$$
Let $\tau\in V$. Then there is a finite $\Gamma\subseteq \omega$ such that
$$\tau\upharpoonright (\omega\sim \Gamma)=
p\upharpoonright (\omega\sim \Gamma).$$
Let $Z=\{i\in \omega: \Gamma\subseteq hi\}$. By the choice of $h$ we have $Z\in F$.
Let $\kappa<\omega$ and $i\in Z$.
We show that $c(\kappa,\epsilon\tau \kappa)_i=\tau \kappa$.
If
$\kappa\in \Gamma,$ then $\kappa\in h(i)$ and so
$c(\kappa,\epsilon \tau \kappa)_i=\tau \kappa$. If $\kappa\notin \Gamma,$
then $\tau \kappa=\kappa$
and $c(\kappa,\epsilon \tau \kappa)_i=\tau\kappa.$
We now prove $(*)$. Let us suppose that $q\in g(a)\cap {\bar{\epsilon}}V$.
Since $q\in \bar{\epsilon}V$ there is an $s\in V$
such that $q=\epsilon\circ s$.
Choose $Z\in F$
such that $$c(\kappa, \epsilon(s\kappa))\supseteq\langle s\kappa: i\in Z\rangle$$
for all $\kappa<\omega$. This is possible by the above.
Let $H=\{i\in \omega: (c^+q)_i\in a\}$.
Then $H\in F$. Since $H\cap Z$ is in $F$
we can choose $i\in H\cap Z$.
Then we have
$$s=\langle s\kappa: \kappa<\omega\rangle=
\langle c(\kappa, \epsilon(s\kappa))_i:\kappa<\omega\rangle=
\langle c(\kappa,q\kappa)_i:\kappa<\omega\rangle=(c^+q)_i\in a.$$
Thus $q\in \epsilon \circ s$. Now suppose that $q=\epsilon\circ s$ with $s\in a$.
Since $a\subseteq V$ we have $q\in \epsilon V$.
Again let $Z\in F$ such that for all $\kappa<\omega$
$$c(\kappa, \epsilon
s \kappa)\supseteq \langle s\kappa: i\in Z\rangle.$$
Then $(c^+q)_i=s\in a$ for all $i\in Z.$ So $q\in g(a).$
Note that $\bar{\epsilon}V\subseteq {}^{\omega}(^{\omega}U/F)$.
Let $rl_{\epsilon(V)}^{\C}$ be the function with domain $\C$
(onto $\bar{\epsilon}(\B))$
such that
$$rl_{\epsilon(V)}^{\C}Y=Y\cap \bar{\epsilon}V.$$
Then we have proved that
$$\bar{\epsilon}=rl_{\bar{\epsilon V}}^{\C}\circ g.$$
It follows that $g$ is a strong sub-base-isomorphism of $\B$ onto $\C$.
\end{proof}

Like the finite dimensional case, we get:
\begin{corollary}\label{v2} $\mathbf{SP}\{ \wp(^{\alpha}U): \text {$U$ a set }\}$ is a variety.
\end{corollary}
\begin{proof} Let $\A\in SA_\alpha$. Then for $a\neq 0$ there exists a weak set algebra $\B$ and $f:\A\to \B$ such that $f(a)\neq 0$.

By the previous theorem there is a set algebra $\C$ such that $\B\cong \C$, via $g$ say. Then $g\circ f(a)\neq 0$, and we are done.
\end{proof}
We then readily obtain:

\begin{corollary}\label{3} Let $\alpha$ be infinite. Then
$$WSA_{\alpha}= {\bf Mod}(\Sigma_{\alpha})={\bf HSP}\{{}\wp(^{\alpha}\alpha^{(Id)})\}={\bf SP}\{\wp(^{\alpha}U): U \text { a set }\}.$$
\end{corollary}
Here we show that the class of subdirect prouct of Pinter's algebras is not a variety, this is not proved by Sagi.
\begin{theorem}\label{notvariety} For infinite ordinals $\alpha$, $RPA_{\alpha}$ is not a variety.
\end{theorem}
\begin{proof} Assume to the contrary that $RTA_{\alpha}$ is a variety and that $RTA_{\alpha}={\bf Mod}(\Sigma_{\alpha})$ for some (countable) 
schema
$\Sigma_{\alpha}.$ Fix $n\geq 2.$ We show that for any set $U$ and any ideal $I$ of $\A=\wp(^nU)$, we have $\A/I\in RTA_n$,
which is not possible since we know that there are relativized set algebras to permutable sets that are not in $RTA_n$.
Define $f:\A\to \wp(^{\alpha}U)$ by $f(X)=\{s\in {}^{\alpha}U: f\upharpoonright n\in X\}$. Then $f$ is an embedding of $\A$ into
$\Rd_n(\wp({}^nU))$, so that we can assume that
$\A\subseteq \Rd_n\B$, for some $\B\in RTA_{\alpha}.$ Let $I$ be an ideal of $\A$, and let $J=\Ig^{\B}I$. Then we claim that
$J\cap \A=I$. One inclusion is trivial; we need to show $J\cap \A\subseteq I$. Let $y\in A\cap J$. Then $y\in \Ig^{\B}I$ and so,
there is a term $\tau$, and $x_1,\ldots x_n\in I$ such that $y\leq \tau(x_1,\dots x_n)$. But $\tau(x_1,\ldots x_{n-1})\in I$ and $y\in A$,
hence $y\in I$, since ideals are closed downwards.
It follows that $\A/I$ embeds into $\Rd_n(\B/J)$
via $x/I\mapsto x/J$. The map is well defined since $I\subseteq J$, and it is one to one, because if $x,y\in A$, such that $x\delta y\in J$,
then $x\delta y\in I$, where $\delta$ denotes symmetric difference.
We have $\B/J\models \Sigma_{\alpha}$.

For $\beta$ an ordinal,
let $K_{\beta}$ denote the class of all full set algebras of dimension $\beta$. Then ${\bf SP}\Rd_nK_{\alpha}\subseteq {\bf SP}K_n$. It is enough to show
that $\Rd_nK_{\alpha}\subseteq {\bf SP}K_n$, and for that it suffices to show that that if $\A\subseteq \Rd_n(\wp({}^{\alpha}U))$,
then $\A$ is embeddable in $\wp(^nW)$, for some set $W$.
Let $\B=\wp({}^{\alpha}U)$. Just  take $W=U$ and define $g:\B\to \wp(^nU)$ by $g(X)=\{f\upharpoonright n: f\in X\}$.
Then $g\upharpoonright \A$ is the desired embedding.
Now let $\B'=\B/I$, then $\B'\in {\bf SP}K_{\alpha}$, so $\Rd_n\B'\in \Rd_n{\bf SP}K_{\alpha}={\bf SP}\Rd_nK_{\alpha}\subseteq {\bf SP}K_n$.
Hence $\A/I\in RTA_n$. But this cannot happen for all $\A\in K_n$ and we are done.
\end{proof}

Next we approach the issue of representations preserving infinitary joins and meets. But first a lemma.

\begin{lemma} Let $\A\in SA_{\alpha}$.
\begin{enumarab}
\item If $X\subseteq \A$, is such that $\sum X=0$ and there exists a representation $f:\A\to \wp(V)$, such that $\bigcap_{x\in X}f(x)=\emptyset$,
then for all $\tau\in {}^{\alpha}{\alpha}^{(Id)}$, $\sum_{x\in X} s_{\tau}x=\emptyset$.
\item  In particular, if $\A$ is completely representable, then for every $\tau\in {}^{\alpha}\alpha^{(Id)}$,
$s_{\tau}$ is completely additive.
\end{enumarab}
\end{lemma}
\begin{proof} Like the finite dimensional case.
\end{proof}
\begin{theorem}\label{counterinfinite} For any $\alpha\geq \omega$, there is an $\A\in  SA_{\alpha}$, and  $S\subseteq \A$, such that $\sum S$
is not preserved by $s_0^1$. In particular, the omitting types theorem fails for our multi-modal logic.
\end{theorem}
\begin{proof}
The second part follows from the previous lemma. Now we prove the first part.
Let  $\B$ be the Stone representation of some atomless Boolean algebra, with unit $U$ in the Stone representation.
Let $$R=\{\times_{i\in \alpha} X_i, X_i\in \B \text { and $X_i=U$ for all but finitely many $i$} \}$$
and
$$A=\{\bigcup S: S\subseteq R: |S|<\omega\}$$
$$S=\{X\times \sim X\times \times_{i>2} U_i: X\in B\}.$$
Then one proceeds exactly like the finite dimensional case, theorem \ref{counter} showing that the sum $\sum S$ is not preserved under $s_0^1$.
\end{proof}
Like the finite dimensional case,  adapting the counterexample to infinite dimensions, we have:

\begin{theorem}\label{counterinfinite2} There is an atomic $\A\in SA_{\alpha}$ such that $\A$ is not completely representable.
\end{theorem}

\begin{proof} First it is clear that if $V$ is any weak space, then $\wp(V)\models \Sigma$.
Let $(Q_n: n\in \omega)$ be a sequence $\alpha$-ary relations such that

\begin{enumroman}

\item $(Q_n: n\in \omega)$ is a partition of
$$V={}^{\alpha}\alpha^{(\bold 0)}=\{s\in {}^{\alpha}\alpha: |\{i: s_i\neq 0\}|<\omega\}.$$



\item Each $Q_n$ is symmetric.
\end{enumroman}
Take $Q_0=\{s\in V: s_0=s_1\}$, and for each $n\in \omega\sim 0$, take $Q_n=\{s\in {}^{\alpha}\omega^{({\bold 0})}: s_0\neq s_1, \sum s_i=n\}.$
(Note that this is a finite sum).
Clearly for $n\neq m$, we have $Q_n\cap Q_m=\emptyset$, and $\bigcup Q_n=V.$
Furthermore, obviously each $Q_n$ is symmetric, that is  $S_{[i,j]}Q_n=Q_n$ for all $i,j\in \alpha$.

Now fix $F$ a non-principal ultrafilter on $\mathcal{P}(\mathbb{Z}^+)$. For each $X\subseteq \mathbb{Z}^+$, define
\[
 R_X =
  \begin{cases}
   \bigcup \{Q_n: n\in X\} & \text { if }X\notin F, \\
   \bigcup \{Q_n: n\in X\cup \{0\}\}      &  \text { if } X\in F
  \end{cases}
\]
Let $$\A=\{R_X: X\subseteq \mathbb{Z}^+\}.$$
Then $\A$ is an atomic set algebra, and its atoms are $R_{\{n\}}=Q_n$ for $n\in \mathbb{Z}^+$.
(Since $F$ is non-principal, so $\{n\}\notin F$ for every $n$.
Then one proceeds exactly as in the finite dimensional case, theorem \ref{counter2}.
\end{proof}

Let $CRSA_{\alpha}$ be the class of completely representable algebras of dimension $\alpha$, then we have

\begin{theorem} For $\alpha\geq \omega$, $CRSA_{\alpha}$ is elementary that is axiomatized by a finite 
schema
\end{theorem}
\begin{proof} Let $\At(x)$ is the formula
$x\neq 0\land (\forall y)(y\leq x\to y=0\lor y=x)$. For distinct $i,j<\alpha$ let $\psi_{i,j}$ be the formula:
$y\neq 0\to \exists x(\At(x)\land s_i^jx\neq 0\land s_i^jx\leq y).$ Let $\Sigma$ be obtained from $\Sigma_{\alpha}$ 
by adding $\psi_{i,j}$
for every distinct $i,j\in \alpha$. These axioms force additivity of the operations $s_i^j$ for every $i,j\in \alpha$. 
The rest is like the finite dimensional 
case.
\end{proof}
The folowing theorem can be easilly destilled from the literature.
\begin{theorem} $SA_{\alpha}$ is Sahlqvist variety, hence it is canonical, $\Str SA_{\alpha}$ is elementary and 
${\bf S}\Cm(\Str SA_{\alpha})=SA_{\alpha}.$
\end{theorem}

We know that if $\A$ is representable on a weak unit, then it is representable on a square one. But for complete representability this is not at all clear,
because the isomorphism defined in \ref{weak} might not preserve arbitrary joins. 
For canonical extensions, we guarantee complete representations.
\begin{theorem} Let $\A\in SA_{\alpha}$. Then $\A^+$ is completely representable on a weak unit.
\end{theorem}
\begin{proof} Let $S$ be the Stone space of $\A$, and for $a\in \A$, let $N_a$ denote the clopen set consisting of all ultrafilters containing $\A$.
The idea is that the operations are completely additive in the canonical extension. Indeed,  for $\tau\in {}^{\alpha}\alpha^{(Id)}$, we have
$$s_{\tau}\sum X=s_{\tau}\bigcup X=\bigcup s_{\tau}X=\sum s_{\tau}X.$$ 
(Indeed this is true for any full complex algebra of an atom structure, and $\A^+=\Cm\Uf \A$.)
In particular, since $\sum \At\A=1$, because $\A$ is atomic, we have $\sum s_{\tau}\At\A=1$, for each $\tau$.
Then we proceed as the finite dimensional case for transposition algebras.
Given any such  $\tau$, let $G(\At \A, \tau)$ be the following no where dense subset of the Stone space of $\A$:
$$G(\At \A, \tau)=S\sim \bigcup N_{s_{\tau}x}.$$
Now given non-zero $a$, let $F$ be a principal ultrafilter generated by an atom below $a$.
Then $F\notin \bigcup_{\tau\in {}^{\alpha}\alpha^{(Id)}} G(\At\A, \tau)$, and the map $h$ defined 
via $x\mapsto \{\tau\in {}^{\alpha}\alpha^{(Id)}: s_{\tau}x\in F\}$, as can easily be checked, 
establishes the complete representation.
\end{proof}
We do not know whether canonical extensions are completely representable on square units.

\begin{theorem} For any finite $\beta$, $\Fr_{\beta}SA_{\alpha}$ is infinite. Furthermore
If $\beta$ is infinite, then $\Fr_{\beta}SA_{\alpha}$ is atomless. In particular, $SA_{\alpha}$ is not locally finite. 
\end{theorem}
\begin{proof}
For the first part, we consider the case when $\beta=1$. Assume that $b$ is the free generator.  First we show that for any finite transposition 
$\tau$ that is not the identity
$s_{\tau}b\neq b$. Let such a $\tau$ be given. Let $\A=\wp(^{\alpha}U)$, and let $X\in \A$, be such that $s_{\tau}X\neq X.$ Such an $X$ obviously
exists. Assume for contradiction that $s_{\tau}b=b$. Let $\B=\Sg^{\A}\{X\}$. 
Then, by freeness, there exists a surjective homomorphism $f:\Fr_{\beta}SA_{\alpha}\to \B$ such that $f(b)=X$.
Hence 
$$s_{\tau}X=s_{\tau}f(b)=f(s_{\tau}b)=f(b)=X,$$
which is impossible. We have proved our claim. Now consider the following subset of $\Fr_{\beta}SA_{\alpha}$, 
$S=\{s_{[i,j]}b: i,j\in \alpha\}$. Then for $i,j,k, l\in \alpha$, with  $\{i,j\}\neq \{k,l\}$, 
we have $s_{[i,j]}b\neq s_{[k,l]}b$, for else, we would get $s_{[i,j]}s_{[k,l]}b=s_{\sigma}b=b $ and $\sigma\neq Id$.
It follows that $S$ is infinite, and so is $\Fr_{\beta}SA_{\alpha}.$ 
The proof for $\beta>1$ is the same.

For the second part, let $X$ be the infinite  generating set. Let $a\in A$ be non-zero. 
Then there is a finite set $Y\subseteq X$ such that $a\in \Sg^{\A} Y$. Let $y\in X\sim Y$.
Then by freeness, there exist homomorphisms $f:\A\to \B$ and $h:\A\to \B$ such that $f(\mu)=h(\mu) $ for all $\mu\in Y$ while
$f(y)=1$ and $h(y)=0$. Then $f(a)=h(a)=a$. Hence $f(a.y)=h(a.-y)=a\neq 0$ and so $a.y\neq 0$ and
$a.-y\neq 0$.
Thus $a$ cannot be an atom.
\end{proof}
For Pinter's algebras the second part applies equally well. For the first part one takes for distict $i,j,k,l$ 
such that $\{i,j\}\cap \{k,l\}=\emptyset$, a relation $X$ in the full set algebra such that $s_i^jX\neq s_k^lX$, and so the set
$\{s_i^jb: i,j\in \alpha\}$ will be infinite, as well.

\section{Adding Diagonals}

We now show that adding equality to our infinite dimensional modal logic, 
algebraically reflected by adding diagonals, does not affect the positive representability results 
obtained
for $SA_{\alpha}$ so far. Also,  in this context, atomicity does not imply complete representability. 
However, we lose elementarity of the class of square completely representable algebras; which is an interesting twist.
We start by defining the concrete algebras, then we provide the finite schema axiomatization.

\begin{definition}

The class of \emph{Representable Diagonal Set Algebras} is defined to  be
$$RDSA_{\alpha}=\mathbf{SP}\{\langle\mathcal{P}(^{\alpha}U); \cap,\sim,S^i_j,S_{ij}, D_{ij}\rangle_{i\neq j\in n}: U\text{ \emph{is a set}},
\}$$
where $S_j^i$ and $S_{ij}$ are as before and $D_{ij}=\{q\in D: q_i=q_j\}$.
\end{definition}

We show that $RDSA_{\alpha}$ is a variety that can be axiomatized by a finite schema.
Let $L_{\alpha}$ be the language of $SA_{\alpha}$ enriched by constants $\{d_{ij}: i,j\in \alpha\}.$

\begin{definition}Let $\Sigma'^d_{\alpha}$ be the axiomatization in $L_{\alpha}$ obtained
by adding to $\Sigma'_{\alpha}$ the following equations for al $i,j<\alpha$.
\begin{enumerate}
\item $d_{ii}=1$
\item $d_{i,j}=d_{j,i}$
\item $d_{i,k}.d_{k,j}\leq d_{i,j}$
\item $s_{\tau}d_{i,j}=d_{\tau(i), \tau(i)}$, $\tau\in \{[i,j], [i|j]\}$.
\end{enumerate}
\end{definition}
\begin{theorem}\label{infinite}
For any infinite ordinal $\alpha$, we have ${\bf Mod}(\Sigma_{\alpha})=RDSA_{\alpha}$.
\end{theorem}
\begin{proof}
Let $\A\in\mathbf{Mod}(\Sigma'^d_{\alpha})$
and let $0^\A\neq a\in A$. We construct a homomorphism $h:\A\longrightarrow\wp (^{\alpha}\alpha^{(Id)})$.
such that $h(a)\neq 0$.
Like before, choose an ultrafilter $\mathcal{F}\subset A$ containing $a$. Let $h:\A\longrightarrow \wp(^{\alpha}\alpha^{(Id)})$
be the following function $h(z)=\{\xi\in ^{\alpha}\alpha^{(Id)}:S_{\xi}^\A(z)\in\mathcal{F}\}.$
The function $h$ respects substitutions but it may not respect the newly added diagonal elements.
To ensure that it does we factor out $\alpha$, the base of the set algebra, by a congruence relation.
Define the following equivalence relation $\sim$ on $\alpha$, $i\sim j$ iff $d_{ij}\in F$. Using the axioms for diagonals $\sim$
is an equivalence relation.
Let $V={}^{\alpha}\alpha^{(Id}),$ and  $M=V/\sim$. For $h\in V$ we write
$h=\bar{\tau}$, if $h(i)=\tau(i)/\sim$ for all $i\in n$. Of course $\tau$ may not be unique.
Now define $f(z)=\{\bar{\xi}\in M: S_{\xi}^{\A}(z)\in \mathcal{F}\}$. We first check that $f$ is well defined.
We use extensively the property $(s_{\tau}\circ s_{\sigma})x=s_{\tau\circ \sigma}x$ for all
$\tau,\sigma\in {}^{\alpha}\alpha^{(Id)}$, a property that can be inferred form our axiomatization.
We show that $f$ is well defined, by induction on the cardinality of
$$J=\{i\in \mu: \sigma (i)\neq \tau (i)\}.$$
Of course $J$ is finite. If $J$ is empty, the result is obvious.
Otherwise assume that $k\in J$. We introduce a piece of notation.
For $\eta\in V$ and $k,l<\alpha$, write
$\eta(k\mapsto l)$ for the $\eta'\in V$ that is the same as $\eta$ except
that $\eta'(k)=l.$
Now take any
$$\lambda\in \{\eta\in \alpha: \sigma^{-1}\{\eta\}= \tau^{-1}\{\eta\}=\{\eta\}\}$$
We have  $${ s}_{\sigma}x={ s}_{\sigma k}^{\lambda}{ s}_{\sigma (k\mapsto \lambda)}x.$$
Also we have (b)
$${s}_{\tau k}^{\lambda}({ d}_{\lambda, \sigma k}. {\sf s}_{\sigma} x)
={ d}_{\tau k, \sigma k} { s}_{\sigma} x,$$
and (c)
$${ s}_{\tau k}^{\lambda}({ d}_{\lambda, \sigma k}.{\sf s}_{\sigma(k\mapsto \lambda)}x)$$
$$= { d}_{\tau k,  \sigma k}.{ s}_{\sigma(k\mapsto \tau k)}x.$$

and (d)

$${ d}_{\lambda, \sigma k}.{ s}_{\sigma k}^{\lambda}{ s}_{{\sigma}(k\mapsto \lambda)}x=
{ d}_{\lambda, \sigma k}.{ s}_{{\sigma}(k\mapsto \lambda)}x$$

Then by (b), (a), (d) and (c), we get,

$${ d}_{\tau k, \sigma k}.{ s}_{\sigma} x=
{ s}_{\tau k}^{\lambda}({ d}_{\lambda,\sigma k}.{ s}_{\sigma}x)$$
$$={ s}_{\tau k}^{\lambda}({ d}_{\lambda, \sigma k}.{ s}_{\sigma k}^{\lambda}
{ s}_{{\sigma}(k\mapsto \lambda)}x)$$
$$={s}_{\tau k}^{\lambda}({ d}_{\lambda, \sigma k}.{s}_{{\sigma}(k\mapsto \lambda)}x)$$
$$= { d}_{\tau k,  \sigma k}.{ s}_{\sigma(k\mapsto \tau k)}x.$$
The conclusion follows from the induction hypothesis.

Clearly $f$ respects diagonal elements.

Now using exactly the technique in theorem \ref{weak} 
(one can easily check that the defined  isomorphism respects diagonal elements), we can square the weak unit, obtaining
the desired result.

\end{proof}

All positive representation theorems, \ref{1}, \ref{2}, \ref{weak}, \ref{v2},  proved for the diagonal free case, hold here.
But the negative do not, because our counter examples {\it do not} contain diagonal elements.

\begin{question} Is it the case, that for $\alpha\geq \omega$, $RDSA_{\alpha}$ is conjuagted, hence  completey additive.
\end{question}
If the answer is affirmative, then we would get all positive results formulated for transposition algebras, given in the second subsection 
of the next section (for infinte dimensions).

In the finite dimensional case, 
we could capture {\it square} complete representability by stipulating that all the operations are completey additive.
However, when we have one single diagonal element, this does not suffice. Indeed, using a simple cardinality argument of Hirsch and Hodkinson,
that fits perfectly here, 
we get the  following slightly surprising result:

\begin{theorem}\label{hh} For $\alpha\geq \omega$, the class of square completely representable algebras is not elementary.
In particular, there is an algebra that is completely representable, but not square completely representable.
\end{theorem}

\begin{proof} \cite{Hirsh}. Let $\C\in SA_{\alpha}$ such that $\C\models d_{01}<1$.  Such algebras exist, for example one can take $\C$ to be 
$\wp(^{\alpha}2).$ Assume that $f: \C\to \wp(^{\alpha}X)$ is a square  complete representation. 
Since $\C\models d_{01}<1$, there is $s\in h(-d_{01})$ so that if $x=s_0$ and $y=s_1$, we have
$x\neq y$. For any $S\subseteq \alpha$ such that $0\in S$, set $a_S$ to be the sequence with 
$ith$ coordinate is $x$, if $i\in S$ and $y$ if $i\in \alpha\sim S$. 
By complete representability every $a_S$ is in $h(1)$ and so in 
$h(\mu)$ for some unique atom $\mu$. 

Let $S, S'\subseteq \alpha$ be destinct and assume each contains $0$.   Then there exists 
$i<\alpha$ such that $i\in S$, and $i\notin S'$. So $a_S\in h(d_{01})$ and 
$a_S'\in h (-d_{01}).$ Therefore atoms corresponding to different $a_S$'s are distinct. 
Hence the number of atoms is equal to the number of subsets of $\alpha$ that contain $0$, so it is at least $^{|\alpha|}2$. 
Now using the downward Lowenheim Skolem Tarski theorem, take an elementary substructure $\B$ of $\C$ with $|\B|\leq |\alpha|.$
Then in $\B$ we have $\B\models d_{01}<1$. But $\B$ has at most $|\alpha|$ atoms, and so $\B$ cannot be {\it square} 
completely representable
(though it is completely representable on a weak unit).
\end{proof}

\section{ Axiomatizating the quasi-varieties}

We start this section by proving a  somewhat general result. 
It is a  non-trival generalisation of S\'agi's result providing an axiomatization for the quasivariety of 
full replacement algebras \cite{sagiphd}, the latter result is obtained by taking $T$ to be the semigroup of all non-bijective maps on $n$.  
 submonoid of $^nn$.
Let $$G=\{\xi\in S_n: \xi\circ \sigma\in T,\text{ for all }\sigma\in T\}.$$

Let $RT_n$ be the class of subdirect products of full set algebras, 
in the similarity type of $T$, and let $\Sigma_n$ be the axiomatization of the variety generated by 
$RT_n$, obtained from any presentation of $T$. (We know that every finite monoid has a representation). 

We now give an axiomatization of the quasivariety $RT_n$, which may not be a variety.

\begin{definition}[The Axiomatization]
For all $n\in\omega,$ $n\geq 2$ the set of quasiequations 
$\Sigma^q_n$ defined to be 
$$\Sigma^q_n=\Sigma_n\cup\{\bigwedge_{\sigma\in T}s_{\sigma}(x_{\sigma})=0\Rightarrow \bigwedge_{\sigma\in T}
s_{i\circ \sigma}(x_{\sigma})=0: i \in  G\}.$$
\end{definition}


\begin{definition}
Let $\A$ be an $RT_n$ like algebra. Let $\xi\in {}^nn$ and let $F$ be an ultrafilter over $\A.$ Then $F_\xi$ denotes the following subset of $A.$
\[
 F_\xi =
  \begin{cases}
  \{t\in A:(\forall \sigma\in T)(\exists a_{\sigma}\in A)S_{\xi\circ\sigma}(a_{\sigma})\in F\\\mbox{ and }
t\geq\bigwedge_{\xi\in G}S_{\eta\circ \sigma} (a_{\sigma}) & \text{ if }\xi\in G\\
   \{a\in A: S^\A_\xi(a)\in F\}      & \text{otherwise }
  \end{cases}
\]
\end{definition}

The proof of the following theorem, is the same as Sagi's corresponding proof for Pinter's algebras \cite{sagiphd}, 
modulo undergoing the obvious replacements; and therefore it will be 
omitted.

\begin{theorem}
\begin{enumarab}
\item Let $\A\in {\bf Mod}(\Sigma^q_n).$ Let $\xi\in {}^nn$ and let $F$ be an ultrafilter over $\A.$ Then, $F_\xi$ is a proper filter over $\A$

\item Let $\A\in {\bf Mod}(\Sigma^q_n)$ and let $F$ be an ultrafilter over $\A.$ Then, $F_{Id}\subseteq F$.

\item $\A\in {\bf Mod}(\Sigma^q_n)$ and let $F$ be an ultrafilter over $\A.$  
For every $\xi\in {}^nn$ let us choose an ultrafilter $F^*_\xi$ containing $F_\xi$ such that $F^*_{Id}=F$. 
Then the following condition holds for this system of ultrafilters: 
$$(\forall \xi\in {}^nn)(\forall \sigma\in T)(\forall a\in A)({S^k_l}^\A(a)\in F^*_\xi\Leftrightarrow a\in F^*_{\xi\circ\sigma})$$
\end{enumarab}
\end{theorem}

\begin{theorem}
For finite $n$, we have ${\bf Mod}(\Sigma^q_n)=RT_n.$
\end{theorem}
\begin{proof} Soundness is immediate \cite{sagiphd}. Now we prove completeness.
Let $\eta\in {}^nn$, and let $a\in A$ be arbitrary. If $\eta=Id$ 
or $\eta\notin G$, then $F_{\eta}$ is the inverse image of $F$ with respect to $s_{\eta}$ and so we are done.
Else $\eta\in G$, and so for all $\sigma\in T$, we have $F_{\eta\circ \sigma}$ is an ultrafilter. 
Let $a_{\sigma}=a$ and $a_f=1$ for $f\in T$ and $f\neq \sigma$.
Now $a_{\tau}\in F$ for all $\tau\in T$. Hence $s_{\eta\circ \tau}(a_{\tau})\in F$, but  
$s_{\eta}a\geq \prod s_{\tau}(a_{\tau})$ 
and we are done. Now with the availabity of $F_{\eta}$ for every $\eta\in {}^nn,$
we can represent our algebra on square units by $f(x)=\{\tau\in {}^nn: x\in F_{\tau}\}$
\end{proof}

\subsection{ Transpositions only}

Consider the case when $T=S_n$ so that we have substitutions corresponding to transpositions. 
This turns out an interesting case with a plathora of positive results, with the sole exception that the class of 
subdirect products of set algebras is {\it not} a 
variety; it is only a quasi-variety. We can proceed exactly like before obtaining a finite equational axiomatization for the variety generated 
by full set algebras by translating 
a presentation of $S_n$. In this case all the operation (corresponding to transpositions) are self-conjugate (because a transposition is the inverse 
of itself) so that our variety, call it $V$, is conjugated, hence completely additive.
Now, the following can be proved exactly like before, undergoing the obvious modifications.
\begin{theorem}
\begin{enumarab}
\item  $V$ is finitely axiomatizable
\item $V$ is locally finite
\item $V$ has the superamalgamation property
\item $V$ is canonical and atom canonical, hence $\At V=\Str V$ is elementary, and finitely axiomatizable.
\item $V$ is closed under canonical extensions and completions
\item $\L_V$ enjoys an omiting types theorem
\item Atomic algebras are completely representable.
\end{enumarab}
\end{theorem}

Here we give a different proof (inspired by the duality theory of modal logic) that $V$ has the the superamalgamation property. 
The proof works verbatim for any submonoid of $^nn$, and indeed, so does the other implemented for $SA_{\alpha}$. In fact the two proofs work for any submonoid of 
$^nn$.

Recall that a frame of type $TA_n$ is a first order structure $\F=(V,  S_{ij})_{i,j\in \alpha}$ where $V$ is an arbitrary set and
and  $S_{ij}$ is a binary relation on $V$  for all $i, j\in \alpha$.

Given a frame $\F$, its complex algebra will be denotet by $\F^+$; $\F^+$ is the algebra $(\wp(\F), s_{ij})_{i,j}$  where for $X\subseteq  V$,
$s_{ij}(X)=\{s\in V: \exists t\in X, (t, s)\in S_{i,j} \}$.
For $K\subseteq TA_n,$ we let $\Str K=\{\F: \F^+\in K\}.$ 

For a variety $V$, it is always the case that 
$\Str V\subseteq \At V$ and equality holds if the variety is atom-canonical. 
If $V$ is canonical, then $\Str V$ generates $V$ in the strong sense, that is 
$V= {\bf S}\Cm \Str V$. For Sahlqvist varieties, as is our case, $\Str V$ is elementary.

\begin{definition}
\item Given a family $(\F_i)_{i\in I}$ of frames, a {\it zigzag  product} of these frames is a substructure of $\prod_{i\in I}\F_i$ such that the
projection maps restricted to $S$ are
onto.
\end{definition}

\begin{definition} Let $\F, \G, \H$ be frames, and $f:\G\to \F$ and $h:\F\to \H$. 
Then $INSEP=\{(x,y)\in \G\times \H: f(x)=h(y)\}$. 
\end{definition}

\begin{lemma} The frame $INSEP \upharpoonright G\times H$ is a zigzag product
of $G$ and $H$, such that $\pi\circ \pi_0=h\circ \pi_1$, where $\pi_0$ and $\pi_1$ are the projection maps.
\end{lemma}
\begin{proof} \cite{Marx} 5.2.4
\end{proof}
For $h:\A\to \B$, $h_+$ denotes the function from $\Uf\B\to \Uf\A$ defined by $h_+(u)=h^{-1}[u]$ 
where the latter is $\{x\in a: h(x)\in u\}.$
For an algebra $\C$, Marx denotes $\C$ by $\C_+,$ and proves:

\begin{theorem}(\cite{Marx} lemma 5.2.6)
Assume that $K$ is a canonical variety and $\Str K$ is closed under finite zigzag products. Then $K$ has the superamalgamation
property.
\end{theorem}
\begin{demo}{Sketch of proof} Let $\A, \B, \C\in K$ and $f:\A\to \B$ and $h:\A\to \C$ be given monomorphisms. 
Then $f_+:\B_+\to \A_+$ and $h_+:\C_+\to \A_+$. We have $INSEP=\{(x,y): f_+(x)=h_+(y)\}$ is a zigzag connection. Let $\F$ be the zigzag product 
of $INSEP\upharpoonright \A_+\times \B_+$. 
Then $\F^+$ is a superamalgam.
\end{demo}
\begin{theorem}
The variety $TA_n$ has $SUPAP$.
\end{theorem}
\begin{proof}  Since $TA_n$ can be easily  defined by positive equations then it is canonical.
The first order correspondents of the positive equations translated to the class of frames will be Horn formulas, hence clausifiable \cite{Marx} theorem 
5.3.5, and so $\Str K$ is closed under finite zigzag products. Marx's theorem finishes the proof.
\end{proof}

The following example is joint with Mohammed Assem (personnel communication).

\begin{theorem}\label{not} For $n\geq 2$, $RTA_n$ is not a variety.
\end{theorem}
\begin{proof}
Let us denote by $\sigma$ the quasi-equation
$$s_f(x)=-x\longrightarrow 0=1,$$ where $f$ is any permutation.
We claim that for all $k\leq n,$ $\sigma$ holds in the small algebra $\A_{nk}$ 
(or more generally, any set algebra with square unit).
This can be seen using a constant map in $^nk.$ More precisely, let $q\in  {}^nk$
be an arbitrary constant map and let $X$ be any subset of $^nk.$
We have two cases for $q$ which are $q\in X$ or $q\in -X$. In either case,
noticing that $q\in X\Leftrightarrow q\in S_f(X),$ it cannot be the case that $S_f(X)=-X.$
Thus, the implication $\sigma$ holds in $\A_{nk}.$
It follows then,
that $RTA_n\models\sigma$ (because the operators $\mathbf{S}$ and $\mathbf{P}$ preserve quasi-equations).

Now we are going to show that there is some element $\B\in PTA_n$, and a specific permutation $f$, such that $\B\nvDash\sigma.$
Let $G\subseteq {}^nn$ be the following permutable set $$G=\{s\in {}^n2:|\{i:s(i)=0\}|=1\}.$$
Let $\B=\wp(G)$, then $\wp(G)\in PTA_n.$ Let $f$ be the permutation defined as follows
For $n=2,3,$ $f$ is simply the transposition $[0,1]$. For larger $n$:
\[
 f =
  \begin{cases}
   [0,1]\circ[2,3]\circ\ldots\circ[n-2,n-1] & \text{if $n$ is even}, \\
   [0,1]\circ[2,3]\circ\ldots\circ[n-3,n-2]      & \text{if $n$ is odd}
  \end{cases}
\]
Notice that $f$ is the composition of disjoint transpositions.
Let $X$ be the following subset of $G,$ $$X=\{e_i:i\mbox{ is odd, }i<n\},$$
where $e_i$ denotes the map that maps every element to $1$ except that the $i$th element is mapped to $0$.
It is easy to see that, for all odd $i<n,$ $e_i\circ f=e_{i-1}.$
This clearly implies that $$S_f^\B(X)=-X=\{e_i:i\mbox{ is even, }i<n\}.$$
Since $0^\B\neq 1^\B,$ $X$ falsifies $\sigma$ in $\B.$ Since $\B\in {\bf H}\{{\wp(^nn)}\}$ we are done.
\end{proof}
Let $Sir(K)$ denote he class of subdirectly indecomposable algebras in $K$.
\begin{corollary} The variety $PTA_n$ is not a discriminator variety
\end{corollary}
\begin{proof} If it were then, there would be a discriminator term on $Sir(RTA_n)$ forcing $RTA_n$
to be  variety, which is not the case.
\end{proof}
All of the above positive results extend to the infinite dimensional case, by using units of the form 
$V=\{t\in {}^{\alpha}\alpha^{(Id)}: t\upharpoonright {\sf sup} t\text { is a bijection }\},$ where ${\sf sup} t=\{i\in \alpha: s_i\neq i\}$, and defining
$s_{\tau}$ for $\tau\in V$ the obvious way. This will be well defined, because the schema axiomatizating 
the variety generated by the square set algebras, is given by lifting the finite axiomatization for $n\geq 5$ (resulting from a presentation of $S_n$, 
to allow indices ranging over $\alpha$. Also, $RTA_{\alpha}$ will {\it not} be a variety using the previous example 
together with the same lifting argument implemented for Pinter's algebras, and finally $TA_{\alpha}$ is not locally finite. 

\section{Decidability}

The decidability of the studied $n$ dimensional multi modal logic, can be proved easily by filtration 
(since the corresponding varieties are locally finite, so such logics are finitely based), or can 
inferred from the decidability of the word problem for finite semigroups. But this is much more than needed.

In fact we shall prove a much stronger result, concerning $NP$ completeness.
The $NP$ completeness of our multi dimensional modal logics (for all three cases, Pinters algebras, transposition algebras, 
 and substitution algebras), by the so called {\it selection method}, which gives a (polynomial) 
bound on a model satifying a given formula in terms of its length. 
This follows from the simple observation that the accessibility relations, are not only partial functions, but they are actually 
total functions, so the method of selection works.  
This for example is {\it not} the case for accessibility relations corresponding to cylindrifiers, and indeed cylindric modal logic of dimension $>2$,
is highly {\it undecidable}, a result of Maddux.

We should also mention that the equational theory of the variety and quasi-varieties (in case of non-closure under homomorphic images)
are also decidable. This is proved exactly like in \cite{sagiphd}, so we 
omit the proof. (Basically, the idea is to reduce the problem to decidability in the finite dimensional using finite reducts).
Our proof of $NP$ completeness is fairly standard. We prepare with some well-known definitions \cite{modal}.
\begin{definition}
Let $\L$ be a normal modal logic, $\M$ a family of finitely based models (based on a $\tau$-frame of finite character). $\L$ \emph{has the polysize model property} with respect to $\M$ if there is a polynomial function $f$ such that  any consistent formula $\phi$ is satisfiable in a model in $\M$ containing at most $f(|\phi|)$ states.
\end{definition}
\begin{theorem}
Let $\tau$ be finite similarity type. Let $\L$ be a consistent normal modal logic over $\tau$ with the polysize 
model property with respect to some class of models $\M$. 
If the problem of deciding whether $M\in\M$ is computable in time polynomial in $|M|$, then $\L$ has an NP-complete satisfiability problem. 
\end{theorem}
\begin{proof}
See Lemma 6.35 in \cite{modal}.
\end{proof}
\begin{theorem}
If $\F$ is a class of frames definable by a first order sentence, 
then the problem of deciding whether $F$ belongs to $\F$ is decidable in time polynomial in the size of $F$.
\end{theorem}
\begin{proof}
See Lemma 6.36 \cite{modal}.
\end{proof}
The same theorem can be stated for models based on elements in $\F$. More precisely,  
replace $\F$ by $\M$ (the class of models based on members of $\F$), and $F$ by $M.$ This is because models are roughly frames with valuations.
We prove our theorem for any submonoid $T\subseteq {}^nn$.
\begin{theorem}
$V_T$ has an NP-complete satisfiability problem.
\end{theorem}
\begin{proof}
By the  two theorems above, it remains  to show that $V_T$ has the polysize model property. We use the {\it selection method.} 
Suppose  $M$ is a model. We define a selection function as follows 
(intuitively, it selects states needed when evaluating a formula in $M$ at $w$): 
$$s(p,w)=\{w\}$$ $$s(\neg\psi,w)=s(\phi,w)$$ $$s(\theta\wedge\psi,w)=s(\theta,w)\cup s(\psi,w)$$
$$s(s_\tau\psi,w)=\{w\}\cup s(\psi,\tau(w)).$$ 
It follows by induction on the complexity of $\phi$ that for all nodes $w$ such that
 $$M,w\Vdash\phi\mbox{ iff }M\upharpoonright s(\phi,w),w\Vdash\phi.$$ 
The new model $M\upharpoonright s(\phi,w)$ has size $|s(\phi,w)|= 1+ \mbox{ the number of modalities in }\phi$. 
This is less than or equal to $|\phi|+1,$ and we done.
\end{proof}


\begin{thebibliography}{}


\bibitem{Andreka} Andr\'eka, H., {\it Complexity of equations valid in algebras of relations}. 
Annals of Pure and Applied logic, {\bf 89}, (1997)  p. 149 - 209.

\bibitem{AGMNS} Andr\'eka, H., Givant, S., Mikulas, S. N\'emeti, I., Simon A., 
{\it Notions of density
that imply representability in algebraic logic.} 
Annals of Pure and Applied logic, {\bf 91}(1998) p. 93 --190. 

\bibitem{AN} H. Andr\'eka, I. N\'emeti {\it Reducing first order logic to $Df_3$ free algebras.}
In Cylindric-like Algebras and Algebraic Logic, 
Bolyai Society Mathematical Studies, Vol. 22  Andréka, Hajnal; Ferenczi, Miklós; Németi, István (Eds.)
(2013.)

\bibitem{modal} Blackburn, P., de Rijke M, Venema, y. {\it Modal Logic} Cambridge test in theoretical Computer Science, Third printing 2008. 




\bibitem{Burris}S. Burris, H.P., Sankappanavar, A course in universal algebra Graduate Texts in Mathematics, Springer Verlag, New York 1981.


\bibitem{f} M. Ferenzci {\it The polyadic generalization of the Boolean axiomatization of field of sets} Trans. Amer Math. society 364 
(2012) 867-886.

\bibitem{Givant} Givant and Venema Y. {\it The preservation of Sahlqvist equations in completions of Boolean algebras with operators}
Algebra Universalis, 41, 47-48 (1999).

\bibitem{step} R. Hirsch and I. Hodkinson {Step-by step building representations in algebraic logic} Journal of Symbolic Logic 
(62)(1)(1997) 225-279.


\bibitem{Hirsh}R. Hirsch and I. Hodkinson, \emph{Complete Representations in Algebraic Logic} Journal of 
Symbolic Logic, 62(3) (1997), pp. 816-847

\bibitem{complete} R. Hirsch and I. Hodkinson {\it Completions and complete representations}
In Cylindric-like Algebras and Algebraic Logic, Bolyai Society Mathematical Studies, 
Vol. 22  Andréka, Hajnal; Ferenczi, Miklós; Németi, István (Eds.)
(2013.)


\bibitem{HMT1} L. Henkin, J.D. Monk and  A.Tarski, {\it Cylindric Algebras Part I}.
North Holland, 1971.

\bibitem{HMT2} L. Henkin, J.D. Monk and  A.Tarski, {\it Cylindric Algebras Part II}.
North Holland, 1985.

\bibitem{h} Hodges {\it Model theory} Cambridge, Encyclopedia of Mathematics.

\bibitem{atomic} Hodkinson, I. {\it A construction of cylindric algebras and polyadic algebras from atomic relation algebras} Algebra Universalis, 
68 (2012) 257-285


\bibitem{Mak} Maksimova, L. 
{\it Amalgamation and interpolation in normal modal logics}.
Studia Logica {\bf 50}(1991) p.457-471.

\bibitem{Marx} M. Marx {\it Algebraic relativization and arrow logic} ILLC Dissertation Series 1995-3, University of Amsterdam, 
(1995).

\bibitem{k} A. Kurusz {\it Represetable cylindric algebras and many dimensional modal logic }
In Cylindric-like Algebras and Algebraic Logic, Bolyai Society Mathematical Studies, 
Vol. 22  And\'reka, Hajnal; Ferenczi, Miklos; N\'emeti, Istvan (Eds.)
(2013.)



\bibitem{semigroup}O. Ganyushkin and V. Mazorchuk,\emph{Classical Finite
Transformation Semigroups-An Introduction}, Springer, 2009.

\bibitem{sagiphd}G. S\'agi,\emph{ A Note on Algebras of Substitutions}, Studia Logica, (72)(2) (2002), p 265-284.

\bibitem{p} G. S\'agi {\it Polyadic algebras} 
In Cylindric-like Algebras and Algebraic Logic, 
Bolyai Society Mathematical Studies, Vol. 22  Andr\'eka, Hajnal; Ferenczi, Miklos; N\'emeti, Istvan (Eds.)
(2013.)
\bibitem{Shelah} Shelah {\it Classification theory, the number of non isomorphic models} North Holland. Studies in Logic and Foundation in Mathematics. (1978)

\bibitem{Tarek}T. Sayed Ahmed {\it Complete representations, completions and omitting types} 
In Cylindric-like Algebras and Algebraic Logic, Bolyai Society Mathematical Studies, Vol. 22  Andr\'eka, Hajnal; Ferenczi, Miklos; 
N\'emeti, Istvan (Eds.)
(2013.)

\bibitem{amal} T. Sayed Ahmed {\it Classes of algebras without the amalgamation property} Logic Journal of IGPL, 192 (2011) p.87-2011.

\bibitem{complete}Sayed Ahmed,T. and Mohamed Khaled {\it On complete representations in algebras of logic}
Logic journal of IGPL {\bf 17}(3)(2009)p. 267-272


\bibitem{v1}  Ide Venema {\it Cylindric modal logic} Journal of Symbolic Logic. (60) 2 (1995)p. 112-198

\bibitem{v2} Ide Venema {\it Cylindric modal logic} 
In Cylindric-like Algebras and Algebraic Logic, Bolyai Society Mathematical Studies, Vol. 22  Andr\'eka, Hajnal; Ferenczi, Miklós; N\'emeti, 
Istvan (Eds.)
(2013.)

\bibitem{v3} Ide Venema{\it Atom structures and Sahlqvist equations} Algebra Universalis, 38 (1997) p. 185-199

\end{thebibliography}
\end{document}